\documentclass[11pt]{amsart}
\usepackage{amsopn}
\usepackage{amssymb, amscd}
\usepackage{multirow}
\usepackage{graphicx, graphics, epsfig}
\usepackage{faktor}
\usepackage{caption}
\usepackage{booktabs}
\usepackage{ textcomp }

\newcommand\restr[2]{{
		\left.\kern-\nulldelimiterspace 
		#1 
		\vphantom{\big|} 
		\right|_{#2} 
}}

\newcommand{\nc}{\newcommand}

\newcommand{\suchthat}{\;\ifnum\currentgrouptype=16 \middle\fi|\;}
\nc{\zg}{\mathfrak{z} } \nc{\ngo}{\mathfrak{n} } \nc{\kg}{\mathfrak{k} }  
\nc{\mg}{\mathfrak{m} } \nc{\bg}{\mathfrak{b} } \nc{\ggo}{\mathfrak{g} } \nc{\hgo}{\mathfrak{h} }
\nc{\ggob}{\overline{\mathfrak{g}} } \nc{\sog}{\mathfrak{so} }
\nc{\sug}{\mathfrak{su} } \nc{\spg}{\mathfrak{sp} } \nc{\slg}{\mathfrak{sl} }
\nc{\glg}{\mathfrak{gl} } \nc{\cg}{\mathfrak{c} } \nc{\rg}{\mathfrak{r} }
\nc{\hg}{\mathfrak{h} } \nc{\tg}{\mathfrak{t} } \nc{\ug}{\mathfrak{u} }
\nc{\dg}{\mathfrak{d} } \nc{\ag}{\mathfrak{a} } \nc{\pg}{\mathfrak{p} }
\nc{\sg}{\mathfrak{s} } \nc{\affg}{\mathfrak{aff} } \nc{\qg}{\mathfrak{q} }

\nc{\pca}{\mathcal{P}} \nc{\nca}{\mathcal{N}} \nc{\lca}{\mathcal{L}}
\nc{\oca}{\mathcal{O}} \nc{\mca}{\mathcal{M}} \nc{\tca}{\mathcal{T}}
\nc{\aca}{\mathcal{A}} \nc{\cca}{\mathcal{C}} \nc{\gca}{\mathcal{G}}
\nc{\sca}{\mathcal{S}} \nc{\hca}{\mathcal{H}} \nc{\bca}{\mathcal{B}}
\nc{\dca}{\mathcal{D}} \nc{\val}{\operatorname{val}}
\nc{\Spec}{Spec}

\nc{\vp}{\varphi} \nc{\ddt}{\frac{d}{dt}} \nc{\dds}{\frac{d}{ds}}
\nc{\dpar}{\frac{\partial}{\partial t}} \nc{\im}{\mathrm{i}}

\nc{\SO}{\mathrm{SO}} \nc{\Spe}{\mathrm{Sp}} \nc{\Sl}{\mathrm{SL}}
\nc{\SU}{\mathrm{SU}} \nc{\Or}{\mathrm{O}} \nc{\U}{\mathrm{U}} \nc{\Gl}{\mathrm{GL}}
\nc{\Se}{\mathrm{S}} \nc{\Cl}{\mathrm{Cl}} \nc{\Spein}{\mathrm{Spin}}
\nc{\Pin}{\mathrm{Pin}} \nc{\G}{\mathrm{GL}_n(\RR)} \nc{\g}{\mathfrak{gl}_n(\RR)}

\nc{\RR}{{\Bbb R}} \nc{\HH}{{\Bbb H}} \nc{\CC}{{\Bbb C}} \nc{\ZZ}{{\Bbb Z}}
\nc{\FF}{{\Bbb F}} \nc{\NN}{{\Bbb N}} \nc{\QQ}{{\Bbb Q}} \nc{\PP}{{\Bbb P}} \nc{\OO}{{\Bbb O}}

\nc{\vs}{\vspace{.2cm}} \nc{\vsp}{\vspace{1cm}} \nc{\ip}{\langle\cdot,\cdot\rangle}
\nc{\ipp}{(\cdot,\cdot)} \nc{\la}{\langle} \nc{\ra}{\rangle} \nc{\unm}{\tfrac{1}{2}}
\nc{\unc}{\tfrac{1}{4}} \nc{\und}{\tfrac{1}{16}} \nc{\no}{\vs\noindent}
\nc{\lam}{\Lambda^2(\RR^n)^*\otimes\RR^n} \nc{\tangz}{{\rm T}^{\rm Zar}}
\nc{\nor}{{\sf n}}  \nc{\mum}{/\!\!/} \nc{\kir}{/\!\!/\!\!/}
\nc{\Ri}{\tfrac{4\Ric_{\mu}}{||\mu||^2}} \nc{\ds}{\displaystyle}
\nc{\ben}{\begin{enumerate}} \nc{\een}{\end{enumerate}} \nc{\f}{\frac}
\nc{\lb}{[\cdot,\cdot]} \nc{\isn}{\tfrac{1}{||v||^2}}
\nc{\gkp}{(\ggo=\kg\oplus\pg,\ip)} \nc{\ukh}{(\ug=\kg\oplus\hg,\ip)}
\nc{\tgkp}{(\tilde{\ggo}=\kg\oplus\pg,\ip)}
\nc{\wt}{\widetilde} \nc{\mm}{M}
\nc{\iop}{\mathtt{i}} \nc{\jop}{\mathtt{j}}

\nc{\Hess}{\operatorname{Hess}} \nc{\ad}{\operatorname{ad}}
\nc{\Ad}{\operatorname{Ad}} \nc{\rank}{\operatorname{rank}}
\nc{\Irr}{\operatorname{Irr}} \nc{\End}{\operatorname{End}}
\nc{\Aut}{\operatorname{Aut}} \nc{\Inn}{\operatorname{Inn}}
\nc{\Aff}{\operatorname{Aff}} \nc{\aff}{\operatorname{aff}}
\nc{\Der}{\operatorname{Der}} \nc{\Ker}{\operatorname{Ker}}
\nc{\Iso}{\operatorname{Iso}} \nc{\Diff}{\operatorname{Diff}}
\nc{\Lie}{\operatorname{L}} \nc{\tr}{\operatorname{tr}} \nc{\dif}{\operatorname{d}}
\nc{\sen}{\operatorname{sen}} \nc{\modu}{\operatorname{mod}}
\nc{\CRic}{\operatorname{PP}} \nc{\Cric}{\operatorname{P}} \nc{\Ricci}{\operatorname{Ric}}
\nc{\sym}{\operatorname{sym}} \nc{\herm}{\operatorname{herm}} \nc{\symac}{\operatorname{sym^{ac}}}
\nc{\symc}{\operatorname{sym^{c}}} \nc{\scalar}{\operatorname{sc}}
\nc{\grad}{\operatorname{grad}} \nc{\ricci}{\operatorname{Rc}}
\nc{\Nor}{\operatorname{Norm}}  \nc{\ricc}{\operatorname{Rc^{c}}}
\nc{\Ricc}{\operatorname{Ric^{c}}} \nc{\ricac}{\operatorname{Rc^{ac}}}
\nc{\Ricac}{\operatorname{Ric^{ac}}} \nc{\Riem}{\operatorname{Rm}}
\nc{\riccig}{\operatorname{ric^{\gamma}}} \nc{\Rin}{\operatorname{M}}
\nc{\Le}{\operatorname{L}} \nc{\tang}{\operatorname{T}}
\nc{\level}{\operatorname{level}} \nc{\rad}{\operatorname{r}}
\nc{\abel}{\operatorname{ab}} \nc{\CH}{\operatorname{CH}} \nc{\Cone}{\operatorname{C}} \nc{\CCone}{\operatorname{CC}}
\nc{\mcc}{\operatorname{mcc}} \nc{\Adj}{\operatorname{Adj}}
\nc{\Order}{\operatorname{O}}  \nc{\inj}{\operatorname{inj}} \nc{\proy}{\operatorname{pr}}
\nc{\vol}{\operatorname{vol}} \nc{\Diag}{\operatorname{Dg}} \nc{\Diagg}{\operatorname{Diag}}
\nc{\Ima}{\operatorname{Im}} \nc{\Rea}{\operatorname{Re}}
\nc{\spann}{\operatorname{span}}

\theoremstyle{plain}
\newtheorem{theorem}{Theorem}[section]
\newtheorem{proposition}[theorem]{Proposition}
\newtheorem{corollary}[theorem]{Corollary}
\newtheorem{lemma}[theorem]{Lemma}
\newtheorem{question}{Question}

\theoremstyle{definition}
\newtheorem{definition}[theorem]{Definition}

\theoremstyle{remark}

\newtheorem{example}[theorem]{Example}

\title[Simply transitive NIL-affine actions of solvable Lie groups]{Simply transitive NIL-affine actions \\ of solvable Lie groups}

\author{Jonas Der\'e} 
\address{KU Leuven Kulak, E. Sabbelaan 53, BE-8500 Kortrijk, Belgium}
\email{jonas.dere@kuleuven.be}

\author{Marcos Origlia}
\address{School of Mathematical Sciences, Monash University, VIC 3800, Australia}
\email{marcosoriglia@gmail.com, marcos.origlia@monash.edu}

\thanks{The first author is supported by a postdoctoral fellowship of the Research Foundation - Flanders (FWO). The second author was supported by the Research Foundation - Flanders (FWO Project G.0F93.17N), CONICET (Argentina) and an ARC (Australia) DP190100317}

\begin{document}

\maketitle

\begin{abstract}
Every simply connected and connected solvable Lie group $G$ admits a simply transitive action on a nilpotent Lie group $H$ via affine transformations. Although the existence is guaranteed, not much is known about which Lie groups $G$ can act simply transitive on which Lie groups $H$. So far the focus was mainly on the case where $G$ is also nilpotent, leading to a characterization depending only on the corresponding Lie algebras and related to the notion of post-Lie algebra structures.

This paper studies two different aspects of this problem. First, we give a method to check whether a given action $\rho: G \to \Aff(H)$ is simply transitive by looking only at the induced morphism $\varphi: \ggo \to \aff(\hgo)$ between the corresponding Lie algebras. Secondly, we show how to check whether a given solvable Lie group $G$ acts simply transitive on a given nilpotent Lie group $H$, again by studying properties of the corresponding Lie algebras. The main tool for both methods is the semisimple splitting of a solvable Lie algebra and its relation to the algebraic hull, which we also define on the level of Lie algebras. As an application, we give a full description of the possibilities for simply transitive actions up to dimension $4$.
\end{abstract}

\tableofcontents

\section{Introduction}

An old question by Milnor \cite{miln77-1} asked whether or not a simply connected and connected (hereinafter called $1$-connected) solvable Lie group $G$ of dimension $n$ admits a representation into the group $\Aff(\RR^n)$, letting $G$ act simply transitively on $\RR^n$. It is well-known that any linear Lie group homeomorphic to $\RR^n$ must be solvable and hence solvable Lie groups are the only possibility for such an action. The answer on Milnor's question is negative, with the first (nilpotent) example exhibited in dimension $11$ by Y.~Benoist in \cite{beno92-1,beno95-1} and later families of nilpotent examples in dimension $10$ by D.~Burde and F.~Grunewald in \cite{burd96-1,bg95-1}. Both constructions were based on the equivalence between simply transitive affine actions on $\RR^n$ and the existence of a complete pre-Lie algebra structure on the corresponding Lie algebra, described in \cite{kim86-1}.

In order to provide a positive answer to Milnor's question a more general setting was considered, namely the one of NIL-affine actions $\rho: G \to \Aff(H)$ where $H$ is a $1$-connected nilpotent Lie group and $\Aff(H) =  H \rtimes \Aut(H)$ is the group of affine transformations on $H$. Indeed, K.~Dekimpe showed in \cite{deki98-1} that for any $1$-connected solvable Lie group $G$, there exist a $1$-connected nilpotent Lie group $H$ and a NIL-affine action $\rho: G \to \Aff(H)$ letting $G$ act simply transitively on $H$. 

This paper studies the related question of determining the pairs $(G,H)$ with $G$ and $H$ $1$-connected Lie groups with $G$ solvable and $H$ nilpotent such that $G$ acts simply transitive on $H$. In the case of $H= \RR^n$, so the usual affine case, one can translate the problem to the Lie algebra level where, as we mentioned before, the existence of a simply transitive affine action of $G$ on $\RR^n$ is equivalent to the existence of a complete pre-Lie algebra structure on the solvable Lie algebra $\ggo$ corresponding to $G$, see \cite{kim86-1}. 

However, in the more general NIL-affine case there is not yet such a correspondence. The particular case when both $G$ and $H$ are nilpotent was studied in \cite{bdd07-1} and the main result gives a complete description on the level of Lie algebras. Note that the action of $G$ on $H$ is completely determined by the corresponding map on the Lie algebras, namely
\begin{align*}
\varphi : \ggo &\to \aff(\hgo) =  \hgo \rtimes \Der(\hgo)  \\
X &\mapsto \left(t(X),D(X) \right)
\end{align*}
with $t: \ggo \to \hgo$ a linear map and $D: \ggo \to \Der(\hgo)$ a morphism of Lie algebras. 
\begin{theorem}{\cite[Theorem 3.1.]{bdd07-1}}
	\label{thm:kd}
Let $G$ and $H$ be $1$-connected nilpotent Lie groups and $\rho: G \to \Aff(H)$ a representation with corresponding maps $t: \ggo \to \hgo$ and $D: \ggo \to \Der(\hgo)$. The group $G$ acts simply transitive on $H$ if and only if the map $t: \ggo \to \hgo$ is a bijection and $D(X)$ is nilpotent for every $X \in \ggo$.
\end{theorem}
\noindent This result was used to show that for every $n \leq 5$ and all $1$-connected nilpotent Lie groups $G$ and $H$ of dimension $n$, there exists a simply transitive NIL-affine action $\rho: G \to \Aff(H)$. 

The purpose of this paper is to fill the gap when the Lie group $G$ is solvable and non-nilpotent. Specifically, we have the following two questions for which we will provide characterizations in terms of the corresponding Lie algebras. The first one asks how we can determine, given the induced map $\varphi$ on the level of Lie algebras, if the corresponding action is simply transitive.

\begin{question}
	\label{question1}
Given a Lie algebra morphism $\varphi: \ggo \to \aff(\hgo)$, determine whether the corresponding NIL-affine action of Lie groups is simply transitive.
\end{question}
\noindent The answer for this question is given in Theorem \ref{thm:q1}, although Example \ref{ex:notDtonly} shows that we cannot express it purely in terms of the maps $t$ and $D$ only. The second question is to investigate on the Lie algebra level whether, given the Lie groups $G$ and $H$, there exists a simply transitive NIL-affine action $\rho: G \to \Aff(H)$.  

\begin{question}
	\label{question2}
Given a solvable Lie algebra $\ggo$ and a nilpotent Lie algebra $\hgo$, determine whether there exists a simply transitive action of the corresponding $1$-connected Lie groups.
\end{question}
\noindent For this, we will describe a criterion in Theorem \ref{cor:q2} depending only on the semisimple splitting of the Lie algebra $\ggo$, which is introduced in Section \ref{sec:back}. It reduces the study of the solvable case to the nilpotent case and additionally embedding the semisimple part in a compatible way. As an application we study the possibilities for dimension $\leq 4$ in Sections \ref{sec:dim3} and \ref{sec:dim4}, based on some computational properties developed in Section \ref{sec:comp}.

\medskip

\noindent \textbf{Acknowledgements.} 
Most of the work has been written during the post-doctoral position at KU Leuven, Campus Kulak Kortrijk of the second named author. He is grateful to the department for the hospitality.

\section{Background}
\label{sec:back}

For proving our main results, we need the theory of linear algebraic groups and the corresponding Lie algebras, which are called algebraic Lie algebras. Since algebraic closures are not always easy to compute, we also introduce the semisimple splitting of solvable Lie algebras, for which we will recall a construction in Section \ref{sec:comp}. To every solvable Lie algebra $\ggo$, we will hence associate two different solvable Lie algebras, namely the semisimple splitting $\ggo^\prime$ and the algebraic hull $\hgo = \overline{\ggo}$, which are related via Theorem \ref{relationssandhull} at the end of this section.

\subsection{Linear algebraic groups}

We start by recalling the structure of linear algebraic groups, as introduced in \cite{hump81-1}. Let $K$ be any subfield of $\CC$, then we define a linear algebraic $K$-group $G$ as a subgroup of $\Gl(n,\CC)$ which is $K$-closed, i.e.~$G$ is the zero set of a finite number of polynomials with coefficients in $K$. A group morphism between two linear algebraic $K$-groups is said to be defined over $K$ (or is an algebraic morphism) if the coordinate functions are given by polynomials over the field $K$. A torus is a linear algebraic $K$-group which is isomorphic to a closed subgroup of diagonal matrices $D(n,\CC)$. 

For our purposes, we will always work with the field of real numbers, so $K = \RR$. We denote by $G(\RR) = G \cap \Gl(n,\RR)$ the subgroup of real points in $G$ and we call $G(\RR)$ a real algebraic group. A real algebraic group automatically has the structure of a Lie group, see \cite[Section 35.3]{hump81-1}, although a Zariski-connected real algebraic group is not necessarily connected as a Lie group, e.g.~the real algebraic group $\Gl(n,\RR)$ is Zariski-connected but as a Lie group it has two connected components. The number of connected components of a Zariski-connected real algebraic group as a Lie group is always finite.

Let $G$ be a connected solvable Lie group which is given as a subgroup $G < \Gl(n, \RR)$.  The real algebraic closure $\overline{G}$ of $G$ in $\Gl(n,\RR)$ is a linear algebraic group which is again solvable and Zariski-connected. Hence $\overline{G}$ splits as a semidirect product $\overline{G} =  U(\overline{G}) \rtimes T$ where $U(\overline{G})$ is the subgroup of $\overline{G}$ consisting of all unipotent elements in $\overline{G}$, called the unipotent radical, and $T$ is a maximal real torus of $\overline{G}$. We will denote $U_G = U(\overline{G})$ since it only depends on $G$. The maximal real torus $T$ is not unique but only unique up to conjugation in $\overline{G}$, see \cite{bore91-1}.

If $H$ is a $1$-connected nilpotent Lie group, then $H$ has a unique structure as a unipotent real algebraic group. Since $\Aut(H)$ is isomorphic to the automorphism group of the corresponding Lie algebra, it carries the structure of a real algebraic group as well. Hence also $\Aff(H) = H \rtimes \Aut(H)$ can be considered as a real algebraic group and therefore we can also consider the real algebraic closure of solvable subgroups of $\Aff(H)$, as we will often do in the remaining part of the paper.

\subsection{Algebraic Lie algebras} 

As was the case for linear algebraic groups, we are mainly interested in Lie algebras over the real numbers $\mathbb{R}$. We call a Lie algebra $\ggo \subset \glg(n,\RR)$ algebraic if there is a real linear algebraic group $G < \Gl(n,\RR)$ such that $\ggo$ is the Lie algebra corresponding to $G$. If $\hgo \subset \glg(n,\RR)$ is any subalgebra, then we call the smallest algebraic Lie algebra $\ggo$ containing $\hgo$ the algebraic closure of $\hgo$ and denote this by $\ggo = \overline{\hgo}$. We say that a Lie algebra morphism between algebraic Lie algebras is algebraic if it corresponds to an algebraic morphism of real algebraic groups. Note that our terminology deviates from the one in \cite{degr17-1}, where the algebraic closure is called an algebraic hull. In this paper, the algebraic hull has another distinct meaning, corresponding to the notion of algebraic hulls for solvable Lie groups as introduced in \cite{ragh72-1}, see Definition \ref{def:hull} below.

Every element $X \in \glg(n,\RR)$ can be uniquely written as $X = X_n + X_s$ where $X_n \in \glg(n,\RR)$ is nilpotent, $X_s \in \glg(n,\RR)$ is semisimple and such that $X_n$ and $X_s$ commute. We call this the (additive) Jordan decomposition of $X$ and $X_n, X_s$ are called the nilpotent and semisimple part of X, respectively. Note that if $\ggo$ is an algebraic Lie algebra, then $\ggo$ contains the nilpotent and semisimple part of all its elements. Moreover, for every solvable algebraic Lie algebra $\ggo$, the set of nilpotent elements of $\ggo$ forms a subalgebra. It is exactly the Lie algebra corresponding to the unipotent radical $U(G)$, where $G$ is the real algebraic group corresponding to $\ggo$, and hence we denote this subalgebra as $\ug(\ggo)$.

Let $\ggo \subset \glg(n,\RR)$ be any solvable Lie algebra and denote by $\overline{\ggo}$ the algebraic closure of $\ggo$ in $\glg(n,\RR)$. The Lie algebra $\overline{\ggo}$ can be decomposed as a semi-direct product $\overline{\ggo} = \ug(\overline{\ggo}) \rtimes \tg $ where $\ug(\overline{\ggo})$ is the subalgebra containing all nilpotent elements of $\overline{\ggo}$ and $\tg$ is an abelian subalgebra consisting of semisimple elements, see \cite[Theorem 4.3.20.]{degr17-1}. Since $\ug(\overline{\ggo})$ depends only on the Lie algebra $\ggo$, we will often denote it by $\ug_\ggo$. Recall that the Lie bracket in the semidirect product $\ug_\ggo \rtimes \tg$ is given by $$ [(u_X,t_X),(u_Y,t_Y)] = \left([u_X,u_Y] + t_X(u_Y) - t_Y(u_X), 0 \right).$$ If the Lie algebra $\overline{\ggo}$ is abelian (or even nilpotent), the natural maps $X \mapsto X_n$ and $X \mapsto X_s$ are algebraic morphisms, see \cite[Theorem 4.7]{bore91-1}. At this point, we want to emphasize that being semisimple or nilpotent is not a property of the Lie algebra itself, but depends on the embedding into $\glg(n,\RR)$. The abelian Lie algebra of dimension $1$ can both be embedded as semisimple or as nilpotent elements, corresponding to the multiplicative and the additive linear algebraic group of dimension $1$.

Since for every $1$-connected nilpotent Lie group $H$ the affine group $\Aff(H)$ is a real algebraic group, the corresponding Lie algebra $\aff(\hgo) = \hgo \rtimes \Der(\hgo)$ is hence an algebraic Lie algebra. We will also consider the algebraic closure of subalgebras in $\aff(\hgo)$ lateron.

\subsection{Algebraic hull}

One of the main tools for studying simply transitive actions of solvable Lie groups $G$ is the notion of the algebraic hull as introduced in \cite{ragh72-1}. Since we focus on Lie algebras here, we will rephrase the definition in terms of Lie algebras.

\begin{definition}
	\label{def:hull}
	Let $\ggo$ be a solvable Lie algebra. We call an algebraic Lie algebra $\hgo$ the \emph{algebraic hull} of $\ggo$ if it satisfies the following conditions:
	\begin{enumerate}
		\item There is an injective Lie algebra morphism $i: \ggo \to \hgo$ such that $\hgo$ is equal to the algebraic closure of $i(\ggo)$;
		\item $\dim(\ug(\hgo)) = \dim(\ggo)$;
		\item The centralizer of $\ug(\hgo)$ is contained in $\ug(\hgo)$.
	\end{enumerate}
\end{definition}
\noindent When talking about the algebraic hull of a Lie algebra $\ggo$, we will slightly abuse notations and consider $\ggo$ as a subalgebra of its own algebraic hull, hence taking $i$ as the inclusion map. In particular, since every algebraic Lie algebra is implicitly embedded in $\glg(n,\RR)$, we will identify $\ggo$ to a subalgebra of $\glg(n,\RR)$ in this case.

Note that the algebraic hull of a Lie algebra $\ggo$ is exactly the Lie algebra corresponding to the algebraic hull of the solvable $1$-connected Lie group $G$ as introduced in \cite{ragh72-1}. The existence of the algebraic hull follows from \cite{ragh72-1}, see also the discussion in \cite[Section 7]{dere19-1} for more details. The inequality $\dim(\ug_{\ggo}) \leq \dim(\ggo)$ always holds, see also Theorem \ref{TcapG} in the next section, so condition $(2)$ means that the nilpotent elements $\ug_\ggo$ of $\overline{\ggo}$ have maximal dimension. By \cite[Lemma 4.41.]{ragh72-1}, it follows that every Lie algebra morphism $\varphi: \ggo \to \hgo$ uniquely extends to an algebraic morphism $\overline{\varphi}: \overline{\ggo} \to \overline{\hgo}$ between the algebraic hulls, implying that the algebraic hull is unique up to algebraic isomorphism. The ideal $\ug_\ggo$ of nilpotent elements is equal to the nilradical of $\hgo$ in this case.

\subsection{Semisimple splitting and nilshadow}

Although \cite{degr17-1} gives an algorithm to compute the algebraic closure of a given algebra $\ggo$, it in general takes some time to find the full description. Hence for the applications we will focus on the semisimple splitting which is enough to decide whether there exists a simply transitive action.

\begin{definition}
	Let $\ggo$ be a solvable Lie algebra. We call a solvable Lie algebra $\ggo^\prime$ a \emph{semisimple splitting} of $\ggo$ if and only if it satisfies the following conditions.
		\begin{enumerate}
		\item $\ggo^\prime =  \ngo \rtimes \tg$ with $\ngo$ the nilradical of $\ggo^\prime$ and $\tg$ an abelian subalgebra;
		\item $\tg$ acts on $\ngo$ via semisimple derivations;
		\item $\ggo$ is an ideal of $\ggo^\prime$ such that $\ggo \cap \tg = 0$ and $\ggo^\prime = \ggo + \tg = \ggo + \ngo$;
		\item $\ngo = \ngo \cap \ggo + \cg$ where $\cg$ is the centralizer of $\tg$ in $\ngo$.
	\end{enumerate}
If $G$ and $G^\prime$ are the unique $1$-connected solvable Lie groups associated to $\ggo$ and $\ggo^\prime$, then we call $G^\prime$ the \emph{semisimple splitting} of $G$.
\end{definition}
\noindent Condition (3) implies that $\ggo^\prime$ is equal to the semidirect product $\ggo \rtimes \tg$. The nilradical of the semisimple splitting of $\ggo$ is called the \textit{nilshadow} of $\ggo$. 

There are many connections between the definitions of the semisimple splitting and the algebraic hull. The following result gives an exact relation between them.

\begin{theorem}
	\label{relationssandhull}
Let $\ggo$ be a solvable Lie algebra with algebraic hull $\hgo$, then the semisimple splitting is equal to the subalgebra $\ggo^\prime = \ggo + \ug(\hgo)$ of $\hgo$. 
\end{theorem}
Although formulated slightly different, this follows from \cite[Proposition 2.4.]{kasu13-1}. This proposition only gives this result for a concrete example, but since both the algebraic hull as the semisimple splitting are unique up to  isomorphisms, the statement holds in general. In particular, the Lie algebra $\ug(\hgo)$ corresponding to the unipotent radical of the algebraic hull of $\ggo$ is isomorphic to the nilradical of the semisimple splitting $\ggo^\prime$, so sometimes we will also refer to the former as the nilshadow. We will recall the construction of the semisimple splitting in Section \ref{sec:comp} which is dedicated to all computations.

\section{Main results}

In this section, we formulate our main results, which answer both Question \ref{question1} and \ref{question2} in a general way. The practicalities and applications then follow in later sections.

Let us start by fixing some notation for the whole section. Note that if the map $\rho: G \to \Aff(H)$ is not injective, the action cannot be simple. So by identifying $G$ with its image $\rho(G)$, we can assume that $G < \Aff(H)$ is a subgroup, or $\ggo \subset \aff(\hgo)$ is a subalgebra. We will assume that $G$ is a connected solvable Lie group with algebraic closure $\overline{G}$ and unipotent radical $U_G = U(\overline{G})$. Since $G$ is connected, the group $\overline{G}$ is Zariski-connected, but it is not necessarily connected as a Lie group, although it only has a finite number of connected components. 

It is well-known that the dimension of $U_G$ is bounded above by the dimension of $G$, see for example \cite[Lemma 4.36.]{ragh72-1}. The following theorem gives a new proof of this fact, including a precise formula for the dimension of $U_G$ depending on the intersection $G \cap T$ with $T$ a maximal torus of $\overline{G}$. Since the dimension of a Lie group and the corresponding Lie algebra are the same, we have an equivalent formulation in terms of Lie algebras $\ggo$ with $\ug_\ggo$ the subalgebra consisting of all nilpotent elements of $\overline{\ggo}$ and the Lie algebra $\tg$ corresponding to a maximal torus.

\begin{theorem}
	\label{TcapG}
	Using the notations of above, we have that $$\dim\left(U_G\right) + \dim(T \cap G) = \dim(G)$$ or equivalently $$\dim\left(\ug_\ggo\right) + \dim(\tg \cap \ggo) = \dim \ggo.$$
\end{theorem}

\begin{proof}We prove this theorem on the level of Lie algebras, which implies the statement for Lie groups. Write $\overline{\ggo} = \ngo \rtimes \tg$ with $\ngo = \ug_\ggo$ and consider the commutator subalgebra $\mg = [\ggo,\ggo] \subset \ngo$ which is also equal to the commutator subalgebra of $\overline{\ggo}$ by \cite[Lemma 4.3.17]{degr17-1}. Take the natural quotient map $$\pi: \overline{\ggo} \to \faktor{\overline{\ggo}}{\mg} =  \faktor{\ngo}{\mg} \rtimes \tg = \faktor{\ngo}{\mg} \oplus \tg$$ which is an algebraic morphism. From $\pi$ we define the linear maps $$\varphi_1: \ggo \to \faktor{\ngo}{\mg}$$ which is the restriction of the projection on the first component and $$\varphi_2: \ker(\varphi_1) \to \tg$$ which is the restriction of the projection on the second component. For every element $X \in \ggo$, we have that  $\varphi_1(X) = \left(X + \mg \right)_n$, so $\varphi_1$ is the restriction of an algebraic morphism $\overline{\ggo} \to \faktor{\ngo}{\mg}$ since $\faktor{\overline{\ggo}}{\mg}$ is an abelian Lie algebra.
	
Because $\overline{\ggo}$ is the algebraic closure of $\ggo$, we get that the algebraic closure of $\varphi_1(\ggo)$ is $\faktor{\ngo}{\mg}$. Since $\faktor{\ngo}{\mg}$ is abelian and consists only of nilpotent elements, every subspace is equal to its own algebraic closure, see \cite[Lemma 4.3.1.]{degr17-1} and \cite[Corollary 4.3.7.]{degr17-1}. Thus we get $\varphi_1(\ggo) = \faktor{\ngo}{\mg}$. Moreover $\ker(\varphi_2) = \mg$ holds and using $\mg \subset \ggo$, we show that $\varphi_2(\ker(\varphi_1)) = \tg \cap \ggo$. Indeed, the inclusion $\tg \cap \ggo \subseteq \varphi_2(\ker(\varphi_1))$ is immediate. For the second inclusion assume $n+t \in \ker (\varphi_1)$ with $n \in \ngo$ and $t \in \tg$, then by the assumption $n + \mg = \varphi_1(n + t) = \mg$, so $n \in \mg \subset \ggo$ and hence also $t = \varphi_2(n+t) \in \ggo \cap \tg$.

The dimension theorem for linear maps implies that 
	\begin{align*}\dim(\ggo) &= \dim(\faktor{\ngo}{\mg}) + \ker(\varphi_1)\\ &= \dim(\ngo) - \dim(\mg) + \dim(\ggo \cap \tg) + \dim(\mg)\\ &= \dim(\ngo) + \dim(\ggo \cap \tg), \end{align*}
	which is exactly the stament of the theorem.
\end{proof}

\noindent For proving our main results, we will mainly need the following consequence.
\begin{corollary}
	\label{cor:noss}
	Using the notations of above, if $G$ is connected and $\dim(G) = \dim(U_G)$, then $G$ is torsion-free and every semisimple element in $G$ is trivial.
\end{corollary}
\begin{proof}
To show that $G$ is torsion-free, we will show that both $[G,G]$ and $\faktor{G}{[G,G]}$ are torsion-free. The first statement is immediate since $[G,G]$ is a subgroup of $U_G$, which is torsion-free. For the second, we write $\faktor{\overline{G}}{[G,G]} = \faktor{U_G}{[G,G]} \oplus T$ for any maximal torus $T$ of $\overline{G}$ exactly as in Theorem \ref{TcapG} and consider the restriction of the projection $\varphi_1: \faktor{G}{[G,G]} \to \faktor{U_G}{[G,G]}$ on the first component. Theorem \ref{TcapG} shows that this map is an isomorphism of groups and hence $\faktor{G}{[G,G]}$ is torsion-free. 

If $g \in G$ is semisimple, then also $g + [G,G]$ is semisimple in $\faktor{\overline{G}}{[G,G]}$. This would imply that $\varphi_1(g) = 0$, or thus that $g \in [G,G] \subset U_G$, leading to $g = e$.
\end{proof}
\noindent In particular, Corollary \ref{cor:noss} holds for the algebraic hull of a solvable Lie algebra. We will return to this remark in Example \ref{ex:ssvshull}, where we consider the algebraic hull of a Lie algebra $\RR^2 \rtimes_D \RR$ with a semisimple derivation $D$, which will hence not be semisimple as an element of the algebraic hull.

Note that under the assumptions of Corollary \ref{cor:noss} the group $G$ must be $1$-connected. It is a well-known fact that if $G$ acts simply transitive on $H$ via affine transformation, then $U_G$ also acts simply transitive on $H$, see for example \cite{bd00-1}. We prove a converse of this statement from Corollary \ref{cor:noss}.

\begin{theorem}
	\label{thm:converse}
	Let $G < \Aff(H)$ be a connected solvable Lie group with $\dim(G) = \dim(H)$ such that $U_G$ acts simply transitive on $H$, then also $G$ acts simply transitive on $H$.
\end{theorem}

\begin{proof}
First we show that the action is simple. Take any $x \in H$ and $g \in G$ such that ${}^g x = x$. Write $T$ for the stabilizer of $x$ in $\overline{G}$, so with $g \in T$. Since the action of $U_G$ is simply transitive, $T \cap U_G = \{e\}$. We show that $ U_G T = \overline{G}$, or thus that $T$ is a maximal torus of $\overline{G}$. Take any $h \in \overline{G}$, then ${}^h x = {}^u x$ for some $u \in U_G$ because the action is transitive. Hence $h = u u^{-1} h $ with $u^{-1} h \in T$ by definition of the stabilizer $T$, so the claim follows. Now $g$, as element of $T$, is semisimple and by Corollary \ref{cor:noss} we get $g=e$. We conclude that the action is simple. 

Next, we show that the action is transitive. Take any maximal torus $T$ of $\overline{G}$, so with $\overline{G} = U_G T$, then we first show that $G T = \overline{G}$. It suffices to show that $U_G \subset G T$, or thus, since $U_G$ is connected, that $\ug_\ggo \subset \ggo + \tg$ with $\tg$ the Lie algebra of $T$. Note that $\tg \cap \ggo = 0$ by Theorem \ref{TcapG} and the fact that $\dim(\ggo) = \dim(\hgo) = \dim(\ug_\ggo)$. Since $\ggo + \tg \subset \overline{\ggo}$ and $\dim(\ggo + \tg ) = \dim(\ggo) + \dim(\tg) =  \dim(\ug_\ggo) + \dim(\tg) = \dim(\overline{\ggo})$, we conclude that $\ggo + \tg = \overline{\ggo}$ and thus $\ug_\ggo \subset \ggo + \tg$. Because $T$ is a maximal torus, it has a fixed point in $H$, see \cite[Lemma 5.3.]{dere19-1}. Since the action of $\overline{G}$ is transitive, this implies that the action of $G$ is transitve.
\end{proof}

We are now ready to give an answer to Question \ref{question1}. Recall that we assume $\varphi: \ggo \to \aff(\hgo)$ to be injective and identify $\ggo$ with its image.

\begin{theorem}
\label{thm:q1}
Let $\hgo$ be a nilpotent Lie algebra and $\ggo \subset \aff(\hgo)$ a solvable subalgebra with algebraic closure $\overline{\ggo}$. The corresponding solvable Lie group $G$ acts simply transitive on $H$ if and only if $\dim(\ggo) = \dim(\hgo)$ and the restriction $\ug_\ggo \subset \aff(\hgo) \to \hgo$ of the natural projection on the first component  is a bijection.
\end{theorem}

\noindent Note that the second condition implies that the unipotent subgroup corresponding to $\ug_\ggo$ acts simply transitive by Theorem \ref{thm:kd}. 
\begin{proof}
First assume that $G$ acts simply transitive on $H$ via affine transformations. The equality $\dim(\ggo) = \dim(G) = \dim(H) = \dim(\hgo)$ is immediate since $G$ and $H$ are diffeomorphic. By \cite[Proposition 2.4.]{bdd07-1} we know that $U_G$ acts simply transitive on $H$. The Lie algebras corresponding to $U_G$ is equal to $\ug_\ggo$ and hence by Theorem \ref{thm:kd} the restriction map $\ug_\ggo \to \hgo$ is a bijection. This concludes the first direction of the proof.

For the other direction, assume that the map $\ug_\ggo \to \hgo$ is a bijection, where again $\ug_\ggo$ is the Lie algebra corresponding to $U_G$. Theorem \ref{thm:kd} implies that $U_G$ acts simply transitive on $H$. By Theorem \ref{thm:converse} we have that $G$ acts simply transitive on $H$, as we need.
\end{proof}

 On how to apply this result for concrete computations, we refer to Section \ref{sec:comp} which gives methods to compute the subalgebra $\ug_\ggo$ starting from some special basis for $\ggo$. Note that we don't have to compute $\tg$ for applying this theorem, which in general is the hardest part of the algebraic closure, we only need to know $\ug_\ggo$.

The next part of this section is dedicated to answering Question \ref{question2}. A first result in this direction uses algebraic embedding of the algebraic hull. In this context, an embedding means that the Lie algebra morphism is injective. 

\begin{theorem}
	\label{thm:q2}
	A $1$-connected solvable Lie group $G$ acts simply transitive on a nilpotent Lie group $H$ via affine transformations if and only if the algebraic hull $\overline{\ggo}$ embeds in $\aff(\hgo)$ as an algebraic subalgebra such that the map $t: \ug_\ggo \to \hgo$ is a bijection. Moreover, if there exists a simply transitive action, we can assume that $\overline{\ggo} \cap \Der(\hgo)$ is the Lie algebra corresponding to a maximal torus of the algebraic hull.
\end{theorem}

\begin{proof}
First assume the action $\rho: G \to \Aff(H)$ is simply transitive. In particular the map $\rho$ is injective and we can identify $G$ and $\rho(G)$. If $G$ acts simply transitive on $H$, \cite[Theorem 5.2.]{dere19-1} implies that $\overline{G} < \Aff(H)$ is the algebraic hull for $G$, where $U(G)$ acts simply transitive. Hence the Lie algebra of $\overline{G}$ is a subalgebra of $\aff(\hgo)$ such that $t: \ug_\ggo \to \ngo$ is a bijection by Theorem \ref{thm:kd}.

For the last part of the theorem, take $T$ the stabilizer of the identity element $e \in H$ in $\overline{G}$, so $T \subset \Aut(H)$. Just as in Theorem \ref{thm:converse}, we have that $T$ is a maximal torus for the group $\overline{G}$. Its Lie algebra $\tg$ is hence a subalgebra of $\Der(\hgo)$. 

For the other direction, assume that the conditions of the theorem hold. Since $\overline{\ggo}$ is an algebraic subalgebra of $\aff(\hgo)$, we also have an embedding $\varphi: \ggo \to \aff(\hgo)$ such that the algebraic closure of $\varphi(\ggo)$ is equal to $\overline{\ggo}$. By Theorem \ref{thm:q1} we get that the action of the corresponding Lie group on $H$ is simply transitive. In particular, the corresponding Lie group is simply connected and hence isomorphic to $G$, leading to a simply transitive action of $G$ on $H$.
\end{proof}

Since the semisimple splitting and the algebraic hull have the same nilradical which is the nilshadow of $\ggo$, one can check the conditions of Theorem \ref{thm:q1} from the semisimple splitting. This avoids all difficulties for computing the algebraic hull and the condition of being an algebraic embedding.

\begin{corollary}
	\label{cor:q2}
	A $1$-connected solvable Lie group $G$ acts simply transitive on a nilpotent Lie group $H$ via affine transformations if and only if the semisimple splitting $\ggo^\prime = \ngo \rtimes \tg$ embeds in $\aff(\hgo)$ such that $\ngo$ consists of nilpotent elements, $\tg$ of semisimple elements and the map $t: \ngo \to \hgo$ is a bijection. Moreover, a simply transitive action exists, we can assume that the semisimple part $\tg$ of $\ggo^\prime$ is a subalgebra of $\Der(\hgo)$.
\end{corollary}
\begin{proof}
One direction, namely if the group $G$ acts simply transitive on $H$, then the semisimple splitting embeds in $\aff(\hgo)$ with the conditions of the corollary, is immediate from a combination of Theorem \ref{relationssandhull} and Theorem \ref{thm:q2}.

For the other direction, consider the semisimple splitting $\ggo^\prime$ as a subalgebra of $\aff(\hgo)$ and hence also $\ggo$ as a subalgebra of $\aff(\hgo)$. First we show that $\overline{\ggo}$ is also the algebraic closure of $\ggo^\prime$, where the latter is equal to $\ngo \rtimes \overline{\tg}$ by \cite[Theorem 4.3.6.]{degr17-1}. For this it suffices to show that $\ngo \subset \overline{\ggo}$, because $\ngo + \ggo = \ggo^\prime$. By condition (4) of the semisimple splitting it is enough to show that $\cg \subset \overline{\ggo}$ with $\cg$ the centralizer of $\tg$ in $\ngo$. This is immediate, since if $X \in \cg$, we can write $X = Y + Z$ with $Y \in \ggo$ and $Z \in \tg$ and thus also $Y = X - Z$. This is exactly the Jordan decomposition of $Y$ because $X$ and $Z$ commute and are respectively nilpotent and semisimple. Since $\overline{\ggo}$ is algebraic it contains the nilpotent and semisimple parts of every element and thus $X \in \overline{\ggo}$. The statement now follows by applying Theorem \ref{thm:q1}.
\end{proof}

We will call a Lie algebra morphism as in Corollary \ref{cor:q2} simply transitive, as it corresponds to a simply transitive action between the corresponding Lie groups. 

Note that this result does not provide us with extra tools to check whether there exists a simply transitive NIL-affine action of one nilpotent Lie group on another. It does show that the study of NIL-affine actions of solvable Lie groups starts with studying the case of nilpotent Lie groups and applying this to the nilshadow $\ug_\ggo$ of the semisimple splitting. Afterwards one needs to check whether the semisimple part $\tg$ can also be embedded in $\Der(\hgo)$, compatible with the embedding of $\ug_\ggo$. 

In the special case where $\hgo$ is the nilshadow of $\ggo$, we always have a simply transitive action.

\begin{corollary}
\label{cor:onnilrad}
Let $G$ be a $1$-connected solvable Lie group, then $G$ acts simply transitive via affine transformations on its nilshadow.
\end{corollary}
\begin{proof}
This follows immediately from Corollary \ref{cor:q2} since if $\ggo^\prime = \ngo \rtimes \tg$ embeds directly into $\aff(\ngo)$ by considering $\tg$ as derivations of $\ngo$. The map $t: \ngo \to \ngo$ is the identity map and hence a bijection.
\end{proof}

For concrete applications, the following lemma about the eigenvalues of derivations will be useful. It will be the main tool for ruling out possibilities lateron. 

\begin{lemma}\label{spec}
Let $\ngo \rtimes \tg$ be a subalgebra of $\aff(\hgo)$ as in Corollary \ref{cor:q2} with $\tg \subset \Der(\hgo)$, then for every $S \in \tg$ it holds that $S$ and $\restr{\ad_S}{ \ngo}$ have the same eigenvalues with the same multiplicities.
\end{lemma}

Here we consider the eigenvalues of $D$ as a derivation on $\hgo$ and $\ad_S$ is the adjoint of $S$ as an element of $\ngo \rtimes \tg$.
\begin{proof}
	By extending the scalars to $\CC$ we can assume that the element $S$ is diagonalizable, so from now on we work with Lie algebras over $\CC$. Let $t: \ngo \to \hgo$ and $D: \ngo \to \Der(\hgo)$ be the natural projections on the first and second component. Now for every eigenvector $X = t(X) + D(X) \in \ngo$ for eigenvalue $\lambda \in \CC$ of $\ad_S$, we have that $$ \lambda (t(X) + D(X)) = [S,t(X) + D(X)] =  \underbrace{S(t(X))}_{\in \ngo} + \underbrace{[S,D(X)]}_{\in \tg}$$ by the definition of the bracket on the semidirect product. This implies that $t(X)$ is an eigenvector of $S$ for the same eigenvalue $\lambda$. Since we can do this for every eigenvector and $t$ is an isomorphism by assumption, the statement follows.
\end{proof}

\section{Computing the nilshadow and the semisimple splitting}
\label{sec:comp}

In this section, we give some comments about how to compute the nilshadow of a solvable Lie algebra, which will be convenient for the applications in the following sections.

Recall that we denote by $X_n$ the nilpotent part in the additive Jordan decomposition of $X \in \glg(n,\RR)$. For applying Theorem \ref{thm:q1} it is crucial to compute the algebra $\ug_{\ggo}$ of a given solvable Lie algebra $\ggo$. Unfortunately, $\ug_\ggo$ is not given by the nilpotent parts $X_n$ of all its elements $X \in \ggo$.

\begin{example}

	Consider the solvable subalgebra $$\ggo = \left\{\begin{pmatrix} 0 & 0 & x \\ 0 & x & y \\ 0 & 0 & 0 \end{pmatrix} \suchthat x, y \in \RR \right\} \subset \glg(n,\RR), $$ which is isomorphic to the solvable non-abelian Lie algebra of dimension $2$. Note that the additive Jordan decomposition of every element $X = X_n + X_s \in \ggo$ is given by \begin{align*} X =  \begin{pmatrix} 0 & 0 & x \\ 0 & x & y \\ 0 & 0 & 0 \end{pmatrix} =  \begin{pmatrix} 0 & 0 & x \\ 0 & 0 & 0 \\ 0 & 0 & 0 \end{pmatrix} +  \begin{pmatrix} 0 & 0 & 0 \\ 0 & x & y \\ 0 & 0 & 0 \end{pmatrix}= X_n + X_s\end{align*} for $x \neq 0$, and $ X = X_n$ otherwise.
Hence the algebraic closure $\overline{\ggo}$ is equal to \begin{align*} \overline{\ggo} = \left\{\begin{pmatrix} 0 & 0 & x \\ 0 & z & y \\ 0 & 0 & 0 \end{pmatrix} \suchthat x, y , z \in \RR \right\} \end{align*} with \begin{align*} \ug_{\ggo} = \left\{\begin{pmatrix} 0 & 0 & x \\ 0 & 0 & y \\ 0 & 0 & 0 \end{pmatrix} \suchthat x, y  \in \RR \right\}.\end{align*} This shows that the map $N: \ggo \to \ug_\ggo: X \mapsto X_n$ is in general not injective nor surjective. Moreover, for a general basis $X_1, \ldots, X_k$ for $\ggo$, one does not have the property that $\left(X_i\right)_n$ spans $\ug_\ggo$.
\end{example}

The previous example demonstrated that in order to compute $\ug_\ggo$ we need some extra tools. 
\begin{proposition}
	\label{propforcomputation} Let $\ggo \subset \glg(n,\RR)$ be a solvable subalgebra with $l = \dim([\ggo,\ggo])$. Consider a basis $X_1, \ldots, X_k$ of $\ggo$ such that the vectors $X_1, \ldots, X_l$ span the commutator subalgebra $[\ggo,\ggo]$ of $\ggo$, then the elements $\left(X_i\right)_n$ span the ideal $\ug_\ggo$ of all nilpotent elements in the algebraic closure $\overline{\ggo}$ of $\ggo$.
\end{proposition}

\begin{proof}
Write $\mg = [\ggo,\ggo]$ for the commutator subalgebra of $\ggo$, then the map $\mg \to \ug_\ggo: X \mapsto X_n$ is a linear map since $\mg$ is nilpotent. In particular, the vectors $\left(X_1\right)_n, \ldots, \left(X_l\right)_n$ span the vector space $\mg$. It suffices to show that the projections of $\left(X_{l+1}\right)_n, \ldots, \left(X_k\right)_n$ span the vector space $\faktor{\ug_\ggo}{\mg}$.

Consider $\overline{\ggo} = \ug_\ggo \rtimes \tg$ with $\tg$ the Lie algebra corresponding to a maximal torus. The ideal $\mg$ is also the commutator subalgebra of $\ggo$ by \cite[Lemma 4.3.17]{degr17-1}. Just as in the proof of Theorem \ref{TcapG}, we consider the projection $\varphi_1: \faktor{\ggo}{\mg} \to \faktor{\ug_\ggo}{\mg}$ on the first component, which is surjective. Since $\varphi_1(X_i) = \left(X_i\right)_n + \mg$, the result now follows.
\end{proof}

In the light of Theorem \ref{thm:q1} this gives us a method for checking whether a given morphism $\varphi:\ggo \to \aff(\hgo)$ corresponds to a simply transitive action. Note that the information does not lie in the seperate maps $t$ and $D$, but in how they interact, as is clear in the following example.

\begin{example}	\label{ex:notDtonly}
Consider the map $\varphi: \RR \to \aff(\RR) = \RR \rtimes \RR: x \mapsto (x,x)$. Note that $\varphi(x)$ is semisimple for every $x \in \RR$ and hence $\ug_{\varphi(\RR)} = 0$ by Proposition \ref{propforcomputation}. This implies that the corresponding affine action of $\RR$ is not simply transitive. Note that $t: \varphi(\RR) \to \hgo$ is bijective for this example.
\end{example}

For many low-dimensional examples the semisimple splitting $\ggo^\prime$ is isomorphic to the algebraic hull $\overline{\ggo}$. We give an example in dimension $3$ where this is not the case. This illustrates why Corollary \ref{cor:q2} is easier to use than Theorem \ref{thm:q2}, since one only needs to compute the nilshadow.

\begin{example}
	\label{ex:ssvshull}
Consider the solvable Lie algebra $ \ggo = \RR^2 \rtimes_D \RR$ where $D$ is the derivation $D$ given by the matrix $$D = \begin{pmatrix} 1 & 0 \\ 0 & \alpha \end{pmatrix}$$ with $\alpha \in \RR$. The semisimple splitting of $\ggo$ is equal to $\ggo \oplus \RR$, but it is algebraic if and only if $\alpha \in \mathbb{Q}$, see \cite[Section 4.4]{degr17-1}. If $\alpha \notin \mathbb{Q}$, then its algebraic hull is isomorphic to $\left(\RR^2 \rtimes_{D_1,D_2} \RR^2\right) \oplus \RR$ with 
\begin{align*}
D_1 &= \begin{pmatrix} 1 & 0 \\ 0 & 0 \end{pmatrix}\\
D_2 &= \begin{pmatrix} 0 & 0 \\ 0 & 1 \end{pmatrix}.
\end{align*}
Note that, although the derivation $D$ is semisimple, it is not semisimple as an element of the algebraic hull, see Corollary \ref{cor:noss}.
\end{example} 
In many cases, the semisimple splitting is easier to compute, since one does not have to take the algebraic closure of a maximal torus. Also, we don't have to worry about embeddings being algebraic as is needed for Theorem \ref{thm:q1}. Therefore Corollary \ref{cor:q2} is the version which is most applicable in practice.

We recall now a direct construction of the semisimple splitting and the nilshadow associated to a given Lie algebra, which we will apply in the next sections. The nilshadow can be determined in two slightly different ways according to \cite[page 73]{dungey03-1} (see also \cite{cf11-1,kasu13-1}).

Let $\ggo$ be a solvable Lie algebra, and $\ngo$ its nilradical. Each derivation $\ad_X$ with $X\in\ggo$ has a Jordan decomposition $\ad_X= (\ad_X)_s + (\ad_X)_n$ with both endomorphisms $\left(\ad_X\right)_s, \left(\ad_Y)\right)_n \in \Der(\ggo)$.
According to \cite[Proposition III.1.1.]{dungey03-1}, there exists a vector space $\mathfrak b$ of $\ggo$ such that $\ggo = \mathfrak b \oplus \ngo$ is the direct sum of vector spaces and for any $X,Y \in \mathfrak b$, it holds that $(\ad_X)_s(Y)=0$. Moreover, the bracket between the semisimple parts satisfies $[\ad(\mathfrak b)_s, \ad(\mathfrak b)_s]=0$.

By considering $\mathfrak b$ as an abelian Lie algebra, the semisimple splitting of $\ggo$ is the semidirect product
$$\ggo'= \mathfrak b\ltimes_{\ad_{(\dots)_s}} \ggo,$$ 
with Lie brackets
$$[(A, X), (B, Y)] = (0, [X, Y] + (\ad_A)_s(Y) - (\ad_B)_s(X)).$$
The nilshadow of $\ggo$, i.~e.~the nilradical of $\ggo'$, is $$\ngo=\{(-X_{\mathfrak b},X) \suchthat  X\in\ggo\},$$
where $X_{\mathfrak b}$ denotes the component of $X$ in $\mathfrak b$.
One can also see the semisimple splitting of $\ggo$ as $$\ggo'= \Ima \ad_s\ltimes \ggo,$$
where $\ad_s:\ggo\to\Der\ggo$.
In this case the nilshadow, which we denote as $\ug$ from now on, can be written as $$\ug=\{X-(\ad_X)_s \in\ggo': X\in\ggo\}.$$ It follows from \cite[Proposition 2.3]{kasu13-2} that $$\ggo'= \Ima \ad_s\ltimes \ug,$$ where $\ug$ is the nilradical of $\ggo^\prime$.

From the latter characterization of the semisimple splitting we can obtain its relation with the algebraic hull as in Theorem \ref{relationssandhull}. We take the Lie group homomorphism $\operatorname{Ad}_s : G \to \Aut(\ug)$ which
corresponds to the Lie algebra homomorphism $\ad_s$. Let $U$ the nilpotent Lie group with Lie algebra $\ug$, then we have
have $\operatorname{Ad}_s(G) \subset \Aut(U)$ where $\Aut(U)$ is an algebraic group. It follows from \cite[Proposition 2.4.]{kasu13-1} that 
$T\ltimes U$ is the algebraic hull of $G$, where $T$ is the Zariski closure of $\operatorname{Ad}_s(G)$ in $\Aut(U)$. 

We demonstrate how to compute the semisimple-splitting $\ggo'$ and the nilshadow for a given Lie algebra. These methods will be implicitly used in the following sections.
\begin{example}\label{example_dim4}
 Consider the $4$-dimensional  Lie algebra $\mathfrak r\mathfrak r_3$ with basis $e_1, e_2, e_3, e_4$ and structure constants given by $[e_1, e_2] = e_2$, $[e_1, e_3] = e_2 +e_3 $ in a basis $\{e_1,e_2,e_3,e_4\}.$ Clearly, the nilradical of $\ggo$ is $\ngo=\text{span}\{e_2,e_3,e_4\}$ and the Jordan decomposition of $\ad_{e_1}|_\ngo$ is 
$$\ad_{e_1}=\begin{pmatrix}

1&1&\\
&1&\\
&&0\\
\end{pmatrix}=\begin{pmatrix}

1&&\\
&1&\\
&&0\\
\end{pmatrix}+
\begin{pmatrix}

0&1&\\
&0&\\
&&0\\
\end{pmatrix}
.$$ 
Therefore, the semisimple splitting of $\mathfrak r\mathfrak r_3$ can be written as $\ggo'=\RR\ltimes_S(\RR\ltimes_T\RR^3)$ where $S=\text{diag}(0,1,1,0)$ and $T$ is the nilpotent part of $\ad_{e_1}$. The structure constants are $[e_0, e_2] = e_2$, $[e_0, e_3] = e_3$, $[e_1, e_3] = e_2$ in a basis $\{e_0,e_1,e_2,e_3,e_4\}.$ It is clear that the nilradical of $\ggo'$, which is the nilshadow of $\ggo$, is $\mathfrak n(\ggo')=\mathfrak r\mathfrak h_3$, and $\ggo'=\RR\ltimes\mathfrak r\mathfrak h_3$ with Lie brackets in the renamed basis $\{e_1,e_2,e_3,e_4,e_5\}$ given by $[e_1, e_3] = e_3$, $[e_1, e_4] = e_4$, $[e_2, e_3] = e_4$.

\end{example}

\section{Simply transitive NIL-affine actions in dimension $3$}
\label{sec:dim3}
In this section we determine whether there exists a simply transitive action for all pairs $(\ggo,\hgo)$ of $3$-dimensional solvable Lie algebras where $\ggo$ is solvable and $\hgo$ is nilpotent. 

As we mentioned before, if $\ggo$ is nilpotent it falls into Theorem \ref{thm:kd} and so it was shown in \cite{bdd07-1} that every possibility occurs up to dimension $5$. If $\hgo$ is abelian, then we have the usual affine case and these possibilities, equivalent to the existence of complete pre-Lie algebra structures, were studied in several papers, see for instance \cite{kim86-1}. Therefore we focus on pairs $(\ggo,\hgo)$ where $\ggo$ is solvable non-nilpotent and $\hgo$ is nilpotent non-abelian. Note that in dimension $2$, the only nilpotent Lie algebra is in fact the abelian one and hence there are no new cases. There is only one nilpotent non-abelian Lie algebra in dimension $3$, the Heisenberg Lie algebra $\hgo_3$, so in the remainder of this section we focus on actions on the Heisenberg group $H_3$. We will only consider the canonical basis $\{f_1,f_2,f_3\}$ of $\hgo_3$ such that $[f_1,f_2]=f_3$. In Table \ref{tab:dim3} we introduce other notations for the $3$-dimensional solvable non-abelian Lie algebras.

\begin{table}[h]

	\centering
	\def\arraystretch{1.4}
	
	\begin{tabular}{|l|l|}
			
		\hline
		Lie algebra &  Lie bracket \\
		\hline
		
		
		$\mathfrak h_3$ & $[e_1, e_2] = e_3$ \\
		\hline
		
		$\mathfrak r_3$ &  $[e_1, e_2] = e_2$, $[e_1, e_3] = e_2 + e_3$ \\			
		\hline	
		
		$\mathfrak r_{3,\lambda}$ &  $[e_1, e_2] = e_2$, $[e_1, e_3] = \lambda e_3$ \\			
		\hline
		
		$\mathfrak r'_{3,\lambda}$ &  $[e_1, e_2] = \lambda e_2-e_3$, $[e_1, e_3] = e_2+\lambda e_3$ \\
		\hline
	\end{tabular}

	\caption{The $3$-dimensional solvable non-abelian Lie algebras.}
			\label{tab:dim3}
\end{table}

As a start, we give four examples of simply transitive actions of solvable Lie groups on the $3$-dimensional Heisenberg group. Afterwards, we show that these are in fact the only possibilities.

\begin{example}\label{r'_{3,0}}
	
	Consider the Lie algebra $\ggo=\mathfrak r'_{3,0}$, which is also known as the Lie algebra $\mathfrak e(2)$ of the group of rigid motions of Euclidean $2$-space. Using the construction mentioned before, it is easy to check that the semisimple splitting of $\ggo$ is isomorphic to $\ggo'= \RR\ltimes \RR^3  = \langle e_1 \rangle \ltimes \langle e_2, e_3, e_4 \rangle$ with Lie brackets $$[e_1,e_2]=-e_3, \quad [e_1,e_3]=e_2 \quad \text{and} \quad e_4\in\zg(\ggo').$$ 
	The Lie algebra $\ggo'$ forms a subalgebra of $\affg(\hg_3)=\hg_3\rtimes\Der(\hg_3)$, by considering the map 
	$\varphi:\ggo'\to \affg(\hg_3)$, where 
	$$\varphi(x_1,x_2,x_3,x_4)=\left( (x_2,x_3,x_4),\begin{pmatrix}
	0&-x_1&0\\
	x_1&0&0\\
	\frac{x_3}{2}&-\frac{x_2}{2}&0
	\end{pmatrix}\right).$$

	\noindent The nilradical of $\ggo'$ is the abelian ideal $\RR^3 = \langle e_2, e_3, e_4 \rangle$ and the restriction of the map $\varphi$ to the nilradical is given by 
	$\restr{\varphi}{\ngo(\ggo^\prime)}(x_2,x_3,x_4)=\left((x_2,x_3,x_4),\begin{pmatrix}
	0&0&0\\
	0&0&0\\
	\frac{x_3}{2}&-\frac{x_2}{2}&0
	\end{pmatrix} \right).$
By Corollary \ref{cor:q2} the corresponding $1$-connected Lie group acts simply transitive on $H_3$.
\end{example}

\begin{example}\label{r_{3,-1}}
	
	Consider the Lie algebra $\ggo=\mathfrak r_{3,-1}$, which is also known as the Lie algebra $\mathfrak e(1, 1)$ of the group of rigid motions of Minkowski $2$-space. Using the construction mentioned before, it is easy to check that the semisimple splitting of $\ggo$ is isomorphic to $\ggo'=\RR\ltimes \RR^3 =  \langle e_1 \rangle \ltimes \langle e_2, e_3, e_4 \rangle  $ with Lie brackets $$[e_1,e_2]=e_2, \quad [e_1,e_3]=-e_3 \quad \text{and} \quad e_4\in\zg(\ggo').$$
The Lie algebra $\ggo'$ forms a subalgebra of $\affg(\hg_3)=\hg_3\rtimes\Der(\hg_3)$ by considering the map 
	$\varphi:\ggo'\to \affg(\hg_3)$, where 
	$$\varphi(x_1,x_2,x_3,x_4)=\left( (x_2,x_3,x_4),\begin{pmatrix}
	x_1&0&0\\
	0&-x_1&0\\
	\frac{x_3}{2}&-\frac{x_2}{2}&0
	\end{pmatrix}\right).$$
	The nilradical of $\ggo'$ is the abelian ideal $\RR^3 = \langle e_2, e_3, e_4 \rangle$ and the restriction of the map $\varphi$ to the nilradical is given by 
$\restr{\varphi}{\ngo(\ggo^\prime)}(x_2,x_3,x_4)=\left((x_2,x_3,x_4),\begin{pmatrix}
	0&0&0\\
	0&0&0\\
	\frac{x_3}{2}&-\frac{x_2}{2}&0
	\end{pmatrix} \right).$
	By Corollary \ref{cor:q2} the corresponding $1$-connected Lie group acts simply transitive on $H_3$.
\end{example}

\begin{example}\label{r_{3,1}}
	
	Consider the Lie algebra $\ggo=\mathfrak r_{3,1}$. 
Using the construction mentioned before, it is easy to check that the semisimple splitting of $\ggo$ is isomorphic to $\ggo'=\RR\ltimes \RR^3 =  \langle e_1 \rangle \ltimes \langle e_2, e_3, e_4 \rangle $ with Lie brackets $$[e_1,e_2]=0, \quad [e_1,e_2]=e_2 \quad \text{and} \quad [e_1,e_3]=e_3.$$
	It is easy to check that $\ggo'$ forms a subalgebra of $\affg(\hg_3)=\hg_3\rtimes\Der(\hg_3)$ by considering the map 
	$\varphi:\ggo'\to \affg(\hg_3)$, where 
	$$\varphi(x_1,x_2,x_3,x_4)=\left( (x_2,x_3,x_4),\begin{pmatrix}
	0&0&0\\
	0&x_1&0\\
	\frac{x_3}{2}&-\frac{x_2}{2}&x_1
	\end{pmatrix}\right).$$
	The nilradical of $\ggo^\prime$ is the abelian ideal $\RR^3 = \langle e_2, e_3, e_4 \rangle$ and the restriction of the map $\varphi$ to the nilradical is given by 
$\restr{\varphi}{\ngo(\ggo^\prime)}(x_2,x_3,x_4)=\left((x_2,x_3,x_4),\begin{pmatrix}
	0&0&0\\
	0&0&0\\
	\frac{x_3}{2}&-\frac{x_2}{2}&0
	\end{pmatrix} \right).$
By Corollary \ref{cor:q2} the corresponding $1$-connected Lie group acts simply transitive on $H_3$.
\end{example}

\begin{example}\label{r_3}
	
	If $\ggo=\mathfrak r_3$, the semisimple splitting is 
	$\ggo'\approx \RR\ltimes \hgo_3 =   \langle e_1 \rangle \ltimes \langle e_2, e_3, e_4 \rangle $ 
	with action given by the diagonal derivation $\operatorname{diag}(0,1,1)$. This case follows immediately from Corollary \ref{cor:onnilrad}, but for completeness we write down the embedding itself. 
	The Lie brackets on $\ggo^\prime$ are $$[e_1,e_2]=0, \quad [e_1,e_3]=e_3 \quad \text{and} \quad [e_3,e_4]=e_4.$$
	It is easy to check that $\ggo'$ forms a subalgebra of $\affg(\hg_3)=\hg_3\rtimes\Der(\hg_3)$, where $\hg_3$ denotes the $3$-dimensional Heisenberg Lie algebra, by considering the map 
	$\varphi:\ggo'\to \affg(\hg_3)$, where 
	$$\varphi(x_1,x_2,x_3,x_4)=\left( (x_2,x_3,x_4),\begin{pmatrix}
	0&0&0\\
	0&x_1&0\\
	0&0&x_1
	\end{pmatrix}\right).$$
	The nilradical of $\ggo'$ is isomorphic to $\hgo_3 = \langle e_2, e_3, e_4 \rangle$ and the restriction of the map $\varphi$ to the nilradical is given by 
$\restr{\varphi}{\ngo(\ggo^\prime)}(x_2,x_3,x_4)=\left((x_2,x_3,x_4),0 \right).$	
	By Corollary \ref{cor:q2} the corresponding $1$-connected Lie group acts simply transitive on $H_3$. 
\end{example}

We now prove that the Lie algebras in the previous examples are the only possibilities by discarding the other examples. The strategy is to study the derivations of the Lie algebra on which we act and apply Lemma \ref{spec}.

\begin{proposition}
	Let $G$ be a solvable non-nilpotent Lie group acting simply transitively on $H_3$ via affine transformations, then the Lie algebra of $G$ is isomorphic to either $\mathfrak r_3$, $\mathfrak r_{3,\lambda}$ with $\lambda \in \{\pm 1\}$ or $\mathfrak r'_{3,0}$.
\end{proposition}

\begin{proof} 
	Let $\ggo$ be the Lie algebra of $G$, and we assume that $G$ acts simply transitive on $H_3$. According to Corollary \ref{cor:q2} there exists a Lie algebra morphism $\varphi: \ggo' \to \affg(\hgo_3)$ with $$\varphi=(t,D): \ggo'=\ngo\rtimes\tg \to \hgo_3\rtimes \Der(\hgo_3)$$  such that $\restr{t}{\ngo}: \ngo \to \hgo_3$ is a bijection. For solvable non-nilpotent Lie algebras of dimension $3$ it holds that $\dim(\tg) = 1$.
    
    It follows from Lemma \ref{spec} that we can take $\varphi$ such that for every $S$ spanning $\tg$ we have $\Spec(\restr{\ad_S}{ \ngo}) = \Spec(D)$ with $D$ some derivation of $\hgo_3$. It is known that $D$ has the form 
    $$    D=\begin{pmatrix}
    a&b&0\\
    c&d&0\\
    m&n&a+d
    \end{pmatrix},$$
    see for example \cite{abdo05-1}. Then the eigenvalues of $D $ are $\{a+d, \frac{a+d\pm \sqrt{\Delta}}{2}\}$, where $\Delta=(a+d)^2-4(ad-bc)$.
 
	First we assume $\ggo=\mathfrak r'_{3,\lambda}$. It can be seen that the semisimple splitting $\ggo'\approx \RR\ltimes \RR^3$ where the action of the semisimple element has eigenvalues $\{0, \lambda+i, \lambda-i \}.$ Since these have to be eigenvalues of $D$, the only possibility is $\lambda=0$.
	
	We assume now $\ggo=\mathfrak r_{3,\lambda}$. In this case the semisimple splitting $\ggo'\approx \RR\ltimes\RR^3$ where the action of the semisimple element has eigenvalues $\{0, 1, \lambda\}$. In the same way as above we have that  $\lambda\in \{\pm 1\}$.
\end{proof}

This proposition shows the strength of Lemma \ref{spec}, since any example which is not excluded by it corresponds to a simply transitive action.

\section{Simply transitive NIL-affine actions in dimension $4$}
\label{sec:dim4}

In this section we study the existence of maps from $\ggo\to \aff(\mathfrak h)$ with $\ggo$ solvable and $\mathfrak h$ nilpotent and $\dim(\ggo) = \dim(\hgo) = 4$, letting $G$ act simply and transitively on $H$ by affine transformations.
There are three $4$-dimensional nilpotent Lie algebras, namely, $\mathfrak{rh}_3, \mathfrak n_4$ and $\RR^4$. The latter corresponds to the usual affine case, already studied in \cite{kim86-1}, therefore we focus on the other two cases. 
Note, if $\ggo$ is nilpotent it falls into Theorem \ref{thm:kd} and so it was studied in \cite{bdd07-1}. 
Therefore we focus on pairs $(\ggo,\hgo)$ where $\ggo$ is solvable non-nilpotent and $\hgo$ is either $\mathfrak{rh}_3$ or $\mathfrak{n}_4$.

We introduce in Table \ref{dim 4} some notation of $4$-dimensional solvable Lie algebras with basis $e_1, e_2, e_3, e_4$.
\begin{table}[h]
	\def\arraystretch{1.5}
	\centering
	\begin{tabular}{l|c}
		$\ggo$ & Lie bracket on basis \\
		\hline
		$\mathfrak r\mathfrak h_3$ & $[e_1, e_2] = e_3$  \\
		
		\hline
		$\mathfrak n_4$ & $[e_1, e_2] = e_3$, $[e_1, e_3] = e_4$ \\
		
		\hline
		
		$\mathfrak r\mathfrak r_3$ & $[e_1, e_2] = e_2$, $[e_1, e_3] = e_2 +e_3 $  \\
		
		\hline
		
		$\mathfrak r\mathfrak r_{3,\lambda}$ & $[e_1, e_2] = e_2$, $[e_1, e_3] = \lambda e_3$, $|\lambda|\leq 1$\\
		
		\hline
		
		$\mathfrak r\mathfrak r'_{3,\lambda}$ & $[e_1, e_2] = \lambda e_2 - e_3$, $[e_1, e_3] = e_2 + \lambda e_3$, $\lambda\geq 0$ \\
		
		\hline
		
		$\mathfrak r_2\mathfrak r_2$  & $[e_1, e_2] = e_2$, $[e_3, e_4] = e_4 $ \\
		
		\hline
		
		$\mathfrak r'_2$  & $[e_1, e_3] = e_3$, $[e_1, e_4] = e_4$, $[e_2, e_3] = e_4$, $[e_2, e_4] = -e_3  $ \\
		
		\hline
		
		$\mathfrak r_4$  & $[e_1, e_2] = e_2$, $[e_1, e_3] = e_2 + e_3$, $[e_1, e_4] = e_3 + e_4$ \\
		
		\hline
		
		$\mathfrak r_{4,\lambda}$  & $[e_1, e_2] = e_2$, $[e_1, e_3] = \lambda e_3$, $[e_1, e_4] = e_3 + \lambda e_4$ \\
		
		\hline
		
		$\mathfrak r_{4,\mu,\lambda}$ & $[e_1, e_2] = e_2$, $[e_1, e_3] = \mu e_3$, $[e_1, e_4] = \lambda e_4$, $\mu\lambda\neq0$, $-1 < \mu \leq \lambda \leq 1$ or $-1=\mu\leq \lambda<0$ \\
		
		\hline
		
		$\mathfrak r'_{4,\gamma,\delta}$ & $[e_1, e_2] = \gamma e_2$, $[e_1, e_3] = \delta e_3 -e_4$, $[e_1, e_4] = e_3 +\delta e_4$, $\gamma>0$ \\
		
		\hline
		
		$\mathfrak d_4$ & $[e_1, e_2] = e_2$, $[e_1, e_3] = -e_3$, $[e_2, e_3] = e_4$ \\
		
		\hline
		
		$\mathfrak d_{4,\lambda}$ & $[e_1, e_2] = \lambda e_2$, $[e_1, e_3] =(1-\lambda) e_3$, $[e_1, e_4] = e_4$, $[e_2, e_3] = e_4$, $\lambda\geq \frac12$ \\
		
		\hline
		
		$\mathfrak d'_{4,\lambda}$ & $[e_1, e_2] = \lambda e_2-e_3$, $[e_1, e_3] = e_2+\lambda e_3$, $[e_1, e_4] =2\lambda e_4$, $[e_2, e_3] = e_4$, $\lambda\geq 0$ \\
		
		\hline 
		
		$\mathfrak h_{4}$ & $[e_1, e_2] = e_2$, $[e_1, e_3] = e_2+e_3$, $[e_1, e_4] = 2e_4$, $[e_2, e_3] = e_4$\\
		
		\hline
	\end{tabular} 
	\caption{The $4$-dimensional solvable non-abelian Lie algebras }
	\label{dim 4}
\end{table}
We computed the semisimple splitting of every $4$-dimensional solvable non-abelian Lie algebra in the same way as we did in Example \ref{example_dim4}. We summarize the result in Table \ref{dim 4 sss}, where we exhibit the Lie brackets of the semisimple splitting and also the nilshadow is indentified. In every case, the basis for $\ggo'$ is $\{e_1,\dots,e_5\}$ except when $\ggo$ is $\mathfrak{r}_2\mathfrak{r}_2$ or $\mathfrak{r}_2'$ where the semisimple splitting has dimension $6$ and the basis is denoted by $\{e_1,\dots,e_6\}$. Note that to go from Table \ref{dim 4} to Table \ref{dim 4 sss}, we sometimes renamed our basis, similarly as in Example \ref{example_dim4}.

\begin{table}[h]
	\def\arraystretch{1.5}
	\centering
	\begin{tabular}{l|c|c|c}
		$\ggo$ & Lie brackets of $\ggo'$ &  $\mathfrak n(\ggo')$ \\
		\hline
		$\mathfrak r\mathfrak h_3$ & $[e_1, e_2] = e_3$  & $\mathfrak r\mathfrak h_3$\\
		
		\hline
		$\mathfrak n_4$ & $[e_1, e_2] = e_3$, $[e_1, e_3] = e_4$   & $\mathfrak n_4$\\
		
		\hline
		
		$\mathfrak r\mathfrak r_3$ & $[e_1, e_3] = e_3$, $[e_1, e_4] = e_4$, $[e_2, e_3] = e_4$  & $\mathfrak r\mathfrak h_3$ \\
		
		\hline
		
		$\mathfrak r\mathfrak r_{3,\lambda}$ & $[e_1, e_3] = e_3$, $[e_1, e_4] = \lambda e_4 $  & $\RR^4$\\
		
		\hline
		
		$\mathfrak r\mathfrak r'_{3,\lambda}$ & $[e_1, e_3] = \lambda e_3 - e_4$, $[e_1, e_4] = e_3 + \lambda e_4 $    & $\RR^4$\\
		
		\hline
		
		$\mathfrak r_2\mathfrak r_2$  & $[e_1, e_4] = e_4$, $[e_2, e_6] = e_6$  & $\RR^4$ \\
		
		\hline
		
		$\mathfrak r'_2$  & $[e_1, e_5] = e_5$, $[e_1, e_6] = e_6$, $[e_2, e_5] = e_6$, $[e_2, e_6] = -e_5$  & $\RR^4$\\
		
		\hline
		
		$\mathfrak r_4$  & $[e_1, e_3] = e_3$, $[e_1, e_4] = e_4$, $[e_1, e_5] = e_5$, $[e_2, e_3] = e_4$, $[e_2, e_4] = e_5$ & $\ngo_4$ \\
		
		\hline
		
		$\mathfrak r_{4,\lambda}$  & $[e_1, e_3] = \lambda e_3$, $[e_1, e_4] = \lambda e_4$, $[e_1, e_5] = e_5$, $[e_2, e_3] = e_4$ & $\mathfrak r\mathfrak h_3$\\
		
		\hline
		
		$\mathfrak r_{4,\mu,\lambda}$ & $[e_1, e_3] = e_3$, $[e_1, e_4] = \mu e_4$, $[e_1, e_5] = \lambda e_5$  & $\RR^4$\\
		
		\hline
		
		$\mathfrak r'_{4,\gamma,\delta}$ & $[e_1, e_3] = \gamma e_3$, $[e_1, e_4] = \delta e_4 -e_5$, $[e_1, e_5] = e_4 +\delta e_5$ & $\RR^4$\\
		
		\hline
		
		$\mathfrak d_4$ & $[e_1, e_2] = e_2$, $[e_1, e_3] = -e_3$, $[e_2, e_3] = e_4$ & $\mathfrak r\mathfrak h_3$\\
		
		\hline
		
		$\mathfrak d_{4,\lambda}$ & $[e_1, e_2] = \lambda e_2$, $[e_1, e_3] =(1-\lambda) e_3$, $[e_1, e_4] = e_4$, $[e_2, e_3] = e_4$ &  $\mathfrak r\mathfrak h_3$\\
		
		\hline
		
		$\mathfrak d'_{4,\lambda}$ & $[e_1, e_2] = \lambda e_2-e_3$, $[e_1, e_3] = e_2+\lambda e_3$, $[e_1, e_4] =2\lambda e_4$, $[e_2, e_3] = e_4$  & $\mathfrak r\mathfrak h_3$\\
		
		\hline 
		
		$\mathfrak h_{4}$ & $[e_1, e_2] = e_2$, $[e_1, e_4] = e_4$, $[e_1, e_5] = 2e_5$, $[e_2, e_3] = e_4$, $[e_2, e_4] = e_5$ & $\mathfrak n_4$\\
		
		\hline
	\end{tabular} 
	\caption{Semisimple splitting and nilshadow of $4$-dimensional solvable Lie algebras }
	\label{dim 4 sss}
\end{table}

\subsection{The case $\hgo=\mathfrak{rh}_3$}

We first consider $\hgo=\mathfrak{rh}_3$ and study the existence of a Lie algebra morphism $\varphi: \ggo' \to \aff (\hgo)$, $\varphi=(t,D)$ satisfying the properties of Corollary \ref{cor:q2}, where $\ggo'$ is the  semisimple splitting associated to a $4$-dimensional solvable Lie algebra $\ggo$.

\begin{proposition}
Let $\hgo$ be the nilpotent Lie algebras $\mathfrak r\mathfrak h_3$ and let $\ggo$ be a solvable Lie algebra with semisimple splitting $\ggo'$ and nilshadow $\mathfrak n(\ggo')=\RR^4$ as in Table \ref{dim 4 sss}. 
	There exists a simply transitive NIL-affine action $\varphi: \ggo \to \aff(\hgo)$ if and only if $\ggo$ is isomorphic to either 
	$\mathfrak{rr}_{3,\lambda}$ with $\lambda \in \left\{0, \pm 1 \right\}$, $\mathfrak{r'}_{4,\gamma,\delta}$ with either $\delta=0$ or $\gamma=2\delta$, $\mathfrak{rr'}_{3,0}$,   
	$\mathfrak{r}_2\mathfrak{r}_2$,
	$\mathfrak{r'}_2$ or 
	$\mathfrak{r}_{4,\mu,\lambda}$ with the following possibilities for the parameters $(\mu,\lambda)$: $(\mu,\lambda)=(-1,\lambda)$ with $-1 \leq \lambda < 0$, $(\mu,\mu)$ with $-1<\mu\leq 1$ and $\mu \neq 0$, $(\mu,-\mu)$ with $-1 < \mu < 0$, $(\mu,1)$ with $-1<\mu<1$ and $\mu \neq 0$, $(\mu,1-\mu)$ with $0<\mu< \frac12$ or $(\mu,1+\mu)$ with $-1<\mu<0$ and $\mu \neq -\frac12$.
	\end{proposition} 

\begin{proof}
 According to Corollary \ref{cor:q2} there exists a Lie algebra morphism $\varphi: \ggo' \to \aff(\mathfrak{rh}_3)$ with $$\varphi=(t,D): \ggo'=\ngo\rtimes\tg \to \mathfrak{rh}_3\rtimes \Der(\mathfrak{rh}_3)$$ such that $t: \ngo \to \mathfrak{rh}_3$ is a bijection. For most possibilities $\tg$ is spanned by one element which we write as $e_1$, otherwise we write a basis as $e_1, e_2$. It follows from Lemma \ref{spec} that we can take $\varphi$ such that $\Spec(\restr{\ad_{e_1}}{ \ngo}) = \Spec(D)$, where $D$ is a derivation of $\mathfrak{rh}_3$. It can be shown that $D$ has the form 
	$$    D=\begin{pmatrix}
	a&b&0&0\\
	c&d&0&0\\
	m&n&a+d&u\\
	p&q&0&v
	\end{pmatrix}.$$
	Then, eigenvalues of $D$ are $\{a+d, \frac{a+d\pm \sqrt{\Delta}}{2},v\}$, where $\Delta=(a+d)^2-4(ad-bc)$. We check when the eigenvalues of $\restr{\ad_{e_1}}{ \ngo}$ coincide with the eigenvalues of $D$ for any Lie algebra in Table \ref{dim 4} whose the nilshadow is $\RR^4$. From now on, we will just write $\ad_{e_1}$ when actually referring to the restriction to $\ngo$.
		
	If $\ggo=\mathfrak{rr}_{3,\lambda}$, the eigenvalues of $\ad_{e_1}$ are $\{0,0,1,\lambda\}$. Comparing this set with the eigenvalues of $D$ we have that $\lambda \in \{ 0, \pm1\}$.

	If $\ggo=\mathfrak{r'}_{4,\gamma,\delta}$, the eigenvalues of $\ad_{e_1}$ are $\{0,\gamma,\delta\pm i\}$. Comparing this set with the eigenvalues of $D$ we obtain two possibilities, namely $\delta = 0$ or $\gamma=2\delta$.
	
	If $\ggo=\mathfrak{rr'}_{3,\lambda}$, the eigenvalues of $\ad_{e_1}$ are $\{0,0,\lambda\pm i\}$. Comparing this set with the eigenvalues of $D$ we have that $\lambda=0$.
	
	If $\ggo=\mathfrak{r}_2\mathfrak{r}_2$, then $\ggo'=\RR^2\ltimes\RR^4$ with the actions given by $\ad_{e_1}=\text{diag}(0,1,0,0)$ and 
	$\ad_{e_2}=\text{diag}(0,0,0,1)$. A computation shows that there cannot be two different commuting derivations with these eigenvalues on $\hgo$.	
%
%

	If $\ggo=\mathfrak{r}'_2$, then $\ggo'=\RR^2\ltimes\RR^4$ with the actions given by $\ad_{e_1}=\text{diag}(0,0,1,1)$ and 
	$\ad_{e_2}=\begin{pmatrix}
	0 & 0 & 0 & 0 \\
	0 & 0 & 0 & 0 \\
	0 & 0 & 0 & -1 \\
	0 & 0 & 1 & 0 \\
	\end{pmatrix}$.	Again, a computation shows that there do not exist two different commuting derivations with these eigenvalues on $\hgo$. 	
	
%
	
	If $\ggo=\mathfrak{r}_{4,\mu,\lambda}$, the eigenvalues of $\ad_S$ are $\{0,1,\mu,\lambda\}$. In this case there are many possible combinations of $\mu$ and $\lambda$, namely the pairs $(-1,\lambda)$ with $-1 \leq \lambda<0$, $(\mu,\mu)$ with $-1<\mu\leq 1$ and $\mu \neq 0$, $(\mu,- \mu)$ with $-1 < \mu < 0$, $(\mu,1)$ with $-1<\mu<1$ and $\mu \neq 0$, $(\mu,1-\mu)$ with $0<\mu<\frac12$ or $(\mu,1+\mu)$ with $-1<\mu<0$ and $\mu \neq -\frac12$.
	
	We exhibit morphisms $D: \ggo' \to \aff (\hgo)$ for the possible cases in Table \ref{Nilshadow-R4} and Table \ref{Nilshadow-R4_continuation}, such that combined with the map $t: \ggo^\prime \to \hgo$ given by $t(x_0,x_1,x_2,x_3,x_4) = (x_1,x_2,x_3,x_4)$ these satisfy the properties of Theorem \ref{thm:q2}. We leave it to the reader to check that these are indeed Lie algebra morphisms. 
\end{proof}

We continue with the case $\hgo=\mathfrak r\mathfrak h_3$, but now we consider Lie algebras in Table \ref{dim 4 sss} whose the nilshadow is $\mathfrak r\mathfrak h_3$. By Corollary \ref{cor:onnilrad} there is a trivial action in these cases, in Table \ref{Nilshadow-h3R} we show the semisimple part for these examples.

\

The last part of the subsection $\hgo=\mathfrak r\mathfrak h_3$ consists of Lie algebras in Table \ref{dim 4 sss} whose nilshadow is equal to $\mathfrak n_4$, namely $\mathfrak{r}_{4}$ and $\mathfrak{h}_{4}$. In both cases a morphism exists, of which we present an example in Table \ref{Nilshadow-n4}.

\subsection{The case $\hgo=\mathfrak n_4$}
We assume from now on $\hgo=\mathfrak n_4$, the filiform $4$-dimensional Lie algebra with Lie brackets given by $[e_1,e_2]=e_3, [e_1,e_3]=e_4$ in the basis $\{e_1,e_2,e_3,e_4\}$. A direct computations shows that a general derivation of $\hgo$ can be writen  as
$$D=\begin{pmatrix}
a & 0 & 0 & 0 \\
b & e & 0 & 0 \\
c & f & a+e & 0 \\
d & g & f & 2a+e \\
\end{pmatrix},$$
and thus has eigenvalues $\{a,e,a+e,2a+e\}$. Note that no derivation of $\hgo$ has eigenvalue $0$ with exactly multiplicity $2$ or $3$.

\begin{proposition}
	Let $\hgo$ be the nilpotent Lie algebras $\mathfrak n_4$ and let $\ggo$ be a solvable Lie algebra with semisimple splitting $\ggo'$ and nilshadow $\mathfrak n(\ggo')=\RR^4$ as in Table \ref{dim 4 sss}. 
There exists a simply transitive action $\varphi: \ggo \to \aff(\hgo)$ if and only if $\ggo$ is isomorphic to 
	$\mathfrak{r}_{4,\mu,\lambda}$ with $\mu=\lambda\in\{-1,\frac12,1\}$ or $\mu=-\frac12$ and $\lambda=\frac12$.
\end{proposition}

\begin{proof}
Assume first that there exists a simply transitive action, then according to Corollary \ref{cor:q2} there exists a Lie algebra morphism $\varphi: \ggo' \to \affg(\hgo)$ with $$\varphi=(t,D): \ggo'=\ngo\rtimes\tg \to \hgo\rtimes \Der(\hgo)$$ such that $t: \ngo \to \hgo$ is a bijection. It follows from Lemma \ref{spec} that we can take $\varphi$ such that $\Spec(\restr{\ad_{e_1}}{ \ngo}) = \Spec(D)$, where $e_1$ is an element of $\tg$ and $D$ is a derivation of $\hgo$. 
	Therefore the eigenvalues of $D$ are of the form $\{a,e,a+e,2a+e\}$.

	We check when the eigenvalues of $\ad_{e_1}$ are equal to the eigenvalues of $D$ for any Lie algebra in Table \ref{dim 4} whose the nilshadow is $\RR^4$. A small computation shows that the only option is $\ggo=\mathfrak{r}_{4,\mu,\lambda}$ with $\mu=\lambda\in\{-1,\frac12,1\}$ or $\mu=-\frac12$ and $\lambda=\frac12$.	We exhibit an example of a map $\varphi: \ggo' \to \aff (\hgo)$ for these possible cases in Table \ref{Nilshadow-R4_continuation}.
\end{proof}

\

We continue with the case $\hgo=\mathfrak n_4$, but now we consider Lie algebras in Table \ref{dim 4} whose the nilshadow is $\mathfrak r\mathfrak h_3$.

\begin{proposition}
	Let $\hgo$ be the nilpotent Lie algebras $\mathfrak n_4$ and let $\ggo$ be a solvable Lie algebra with semisimple splitting $\ggo'$ and nilshadow $\mathfrak n(\ggo')=\mathfrak r\mathfrak h_3$ as in Table \ref{dim 4 sss}. 
	There exists a simply transitive action $\varphi: \ggo \to \aff(\hgo)$ if and only if $\ggo$ is isomorphic to 
	$\mathfrak{r}_{4,\lambda}$ with $\lambda \in \left\{-1,\frac 12, 1\right\}$, or
	$\mathfrak{d}_{4,\lambda}$ with $\lambda \in \left\{\frac 12, 2\right\}$.
\end{proposition}

\begin{proof}
	Assume first that there exists a simply transitive action, so there exists a map $\varphi: \ggo^\prime \to \aff(\hgo)$ of the semisimple splitting satisfying the conditions of Corollary \ref{cor:q2}. It follows from Lemma \ref{spec} that we can take $\varphi$ such that $\Spec(\restr{\ad_{e_1}}{ \ngo}) = \Spec(D)$, where $e_1$ is an element of $\tg$ and $D$ is a derivation of $\hgo_3$. 
	Therefore the eigenvalues of $D$ are of the form $\{a,e,a+e,2a+e\}$.

	We check when the eigenvalues of $\ad_{e_1}$ are equal to the eigenvalues of $D$ for any Lie algebra in Table \ref{dim 4} whose the nilshadow is $\mathfrak{rh}_3$. It can be proved that the only options are
	$\mathfrak{r}_{4,\lambda}$ with $\lambda \in \left\{-1,\frac 12, 1\right\}$, or
	$\mathfrak{d}_{4,\lambda}$ with $\lambda \in \left\{\frac 12, 2\right\}$.
	We exhibit an example of a map $\varphi: \ggo' \to \aff (\hgo)$ for the possible cases, see Table \ref{Nilshadow-h3R}.
\end{proof}

The last part of the case $\hgo=\mathfrak n_4$ consists of Lie algebras in Table \ref{dim 4} whose the nilshadow is $\mathfrak n_4$. In this case there is a trivial Lie algebra morphism by Corollary \ref{cor:onnilrad}. For completeness we give the compatible derivation $D_X$ for $X\in\tg$ in Table \ref{Nilshadow-n4}.

We gave existence results of simply transitive actions from $\ggo$ to $\hgo$ for all possible $4$-dimensional Lie algebras $\ggo$ and $\hgo$, where $\ggo$ is solvable and $\hgo$ is nilpotent. An overview is given in Table \ref{table:overview}, where we list all possibilities. From the explicit morphisms from $\ggo'$ to $\hgo$ exhibited in Tables \ref{Nilshadow-R4}, \ref{Nilshadow-R4_continuation}, \ref{Nilshadow-h3R} and \ref{Nilshadow-n4} we can obtain morphisms from $\ggo$ to $\hgo$. 
As an example, let us recall notation from Example \ref{example_dim4}. Consider the Lie algebra $\ggo=\mathfrak r\mathfrak r_3$ with basis $e_1, e_2, e_3, e_4$ and structure constants given by $[e_1, e_2] = e_2$, $[e_1, e_3] = e_2 +e_3 $ in a basis $\{e_1,e_2,e_3,e_4\}.$ From Table \ref{Nilshadow-h3R} we have a morphism $\varphi: \ggo' \to \aff(\hgo_3\oplus\RR)$ given by 
$\varphi(x_0,x_1,x_2,x_3,x_4)=( (x_1,x_2,x_3,x_4), \text{diag}(0,x_0,x_0,0) )$ in the basis $f_1, f_2, f_3, f_4$ of $\mathfrak r\mathfrak r_3$ with constant structures given by $[f_1,f_2]=f_3$. Then, by composing with the inclusion of $\ggo$ into its semisimple splitting $\ggo^\prime$, we get a Lie algebra morphism 
$\psi: \mathfrak r\mathfrak r_3 \to \aff(\hgo_3\oplus\RR)$ given by 
$$\psi(x_1,x_2,x_3,x_4)=\left( (x_1,x_3,x_2,x_4), \text{diag}(0,x_1,x_1,0) \right)$$ which induces a simply transitive action between the corresponding $1$-connected Lie groups.
        
\

	\begin{table}[h!]
	\def\arraystretch{1.5}
	\centering
	\begin{tabular}{l|c|c}
		$\ggo$ & $\Der(\hgo_3\oplus\RR)$ & $\Der(\ngo_4)$ \\
		\hline
		
		$\mathfrak r\mathfrak r_{3,-1}$  &  $\begin{pmatrix}
	x_0 & 0 & 0 & 0 \\
	0 & -x_0 & 0 & 0 \\
	\frac{x_2}{2} & -\frac{x_1}{2} & 0 & 0 \\
	0 & 0 & 0 & 0 \\
	\end{pmatrix}$ & \texttimes \\
		
		\hline
		
		$\mathfrak r\mathfrak r_{3,1}$  & $\begin{pmatrix}
	0 & 0 & 0 & 0 \\
	0 & x_0 & 0 & 0 \\
	\frac{x_2}{2} & -\frac{x_1}{2} & x_0 & 0 \\
	0 & 0 & 0 & 0 \\
	\end{pmatrix}$ & \texttimes \\
		
		\hline
		
		$\mathfrak r\mathfrak r_{3,0}$  &  $\begin{pmatrix}
	0 & 0 & 0 & 0 \\
	0 & 0 & 0 & 0 \\
	\frac{x_2}{2} & -\frac{x_1}{2} & 0 & 0 \\
	0 & 0 & 0 & x_0 \\
	\end{pmatrix}$ & \texttimes \\
		
		\hline

		$\mathfrak r'_{4,\gamma,0}$ &  $\begin{pmatrix}
	0 & x_0 & 0 & 0 \\
	-x_0 & 0 & 0 & 0 \\
	\frac{x_2}{2} & -\frac{x_1}{2} & 0 & 0 \\
	0 & 0 & 0 & \gamma x_0 \\
	\end{pmatrix}$ & \texttimes \\
		
		\hline
		
		$\mathfrak r'_{4,2\delta,\delta}$ &  $\begin{pmatrix}
	\delta x_0 & x_0 & 0 & 0 \\
-x_0 & \delta x_0 & 0 & 0 \\
	\frac{x_2}{2} & -\frac{x_1}{2} & 2 \delta x_0 & 0 \\
	0 & 0 & 0 & 0 \\
	\end{pmatrix}$ & \texttimes \\
		
		\hline
				$\mathfrak r\mathfrak r'_{3,0}$ &  $\begin{pmatrix}
		0 & x_0 & 0 & 0 \\
		-x_0 & 0 & 0 & 0 \\
		\frac{x_2}{2} & -\frac{x_1}{2} & 0 & 0 \\
		0 & 0 & 0 & 0 \\
		\end{pmatrix}$ & \texttimes\\
		
		\hline 
		
		\bottomrule
	\end{tabular} 
	\caption{Simply transitive actions for $\ggo$ with nilshadow $\mathfrak n(\ggo')=\RR^4$, part I.}
	\label{Nilshadow-R4}
\end{table}
\

	\begin{table}[h!]
	\def\arraystretch{1.5}
	\centering
	\begin{tabular}{l|c|c}
		$\ggo=\mathfrak r_{4,\mu,\lambda}$ &  $\Der(\hgo_3\oplus\RR)$ & $\Der(\ngo_4)$ \\

		\hline
		
		$\mathfrak r_{4,-1,\lambda}$ $-1\leq\lambda<0$  & $\begin{pmatrix}
	x_0 & 0 & 0 & 0 \\
	0 & -x_0 & 0 & 0 \\
	\frac{x_2}{2} & -\frac{x_1}{2} & 0 & 0 \\
	0 & 0 & 0 & \lambda x_0 \\
	\end{pmatrix}$ & \text{only} $\lambda=-1$: 
	$\begin{pmatrix}
	-x_0 & 0 & 0 & 0 \\
	0 & x_0 & 0 & 0 \\
	0 & -x_1 & 0 & 0 \\
	0 & 0 & -x_1 & -x_0 \\
	\end{pmatrix}$\\
		
		\hline
		
		$\mathfrak r_{4,\mu,\mu}$ $-1<\mu\leq1, \mu \neq 0$ &  $\begin{pmatrix}
	0 & 0 & 0 & 0 \\
	0 & \mu x_0 & 0 & 0 \\
	\frac{x_2}{2} & -\frac{x_1}{2} & \mu x_0 & 0 \\
	0 & 0 & 0 & x_0 \\
	\end{pmatrix}$ & \text{only} $\mu=\frac12$:
	$\begin{pmatrix}
	\frac{x_0}{2} & 0 & 0 & 0 \\
	0 & 0 & 0 & 0 \\
	0 & -x_1 & \frac{x_0}{2}  & 0 \\
	0 & 0 & -x_1 & x_0 \\
	\end{pmatrix}$\\
	
  & & \text{and} $\mu=1$:
	$\begin{pmatrix}
	0 & 0 & 0 & 0 \\
	0 & x_0 & 0 & 0 \\
	0 & -x_1 & x_0 & 0 \\
	0 & 0 & -x_1 & x_0 \\
	\end{pmatrix}$\\
		
		\hline
		
		$\mathfrak r_{4,\mu,-\mu}$, $-1 < \mu < 0$ &  $\begin{pmatrix}
	 \mu\frac{x_0}{2} & 0 & 0 & 0 \\
	0 &  -\mu\frac{x_0}{2} & 0 & 0 \\
	\frac{x_2}{2} & -\frac{x_1}{2} & 0 & 0 \\
	0 & 0 & 0 & x_0 \\
	\end{pmatrix}$ & $\mu=-\frac12$: $\begin{pmatrix}
	-\frac{x_0}2 & 0 & 0 & 0 \\
	0 & x_0 & 0 & 0 \\
	0 & -x_1 & \frac{x_0}2  & 0 \\
	0 & 0 & -x_1 & 0\\
	\end{pmatrix}$\\
	
		\hline
		
		$\mathfrak r_{4,\mu,1}$ $-1 < \mu <1, \mu \neq 0$ &  $\begin{pmatrix}
	x_0 & 0 & 0 & 0 \\
	0 & 0 & 0 & 0 \\
	\frac{x_2}{2} & -\frac{x_1}{2} &  x_0 & 0 \\
	0 & 0 & 0 & \mu x_0 \\
	\end{pmatrix}$ & \texttimes\\
	
		\hline
		
		$\mathfrak r_{4,\mu,1-\mu}$ $0<\mu<\frac12$ & $\begin{pmatrix}
	\mu x_0 & 0 & 0 & 0 \\
	0 & (1-\mu)x_0 & 0 & 0 \\
	\frac{x_2}{2} & -\frac{x_1}{2} & x_0 & 0 \\
	0 & 0 & 0 & 0 \\
	\end{pmatrix}$ & \texttimes\\
		
		\hline

    $\mathfrak r_{4,\mu,1+\mu}$ $-1<\mu<0$ $\mu\neq-\frac12$ &  $\begin{pmatrix}
	x_0 & 0 & 0 & 0 \\
	0 & \mu x_0 & 0 & 0 \\
	\frac{x_2}{2} & -\frac{x_1}{2} & (1+\mu)x_0 & 0 \\
	0 & 0 & 0 & 0 \\
	\end{pmatrix}$ & \texttimes \\
		
		\hline
		
		\bottomrule
	\end{tabular} 
	\caption{Simply transitive action for $\ggo$ with nilshadow $\ngo(\ggo^\prime) = \RR^4$, part II.}
	\label{Nilshadow-R4_continuation}
\end{table}

\

\begin{table}[h!]
	\def\arraystretch{1.47}
	\centering
	\begin{tabular}{l|c|c}
		$\ggo$ &  $\Der(\hgo_3\oplus\RR)$ & $\ngo_4\rtimes\Der(\ngo_4)$ \\
		\hline		
				
		$\mathfrak r\mathfrak r_3$  &  $\text{diag}(0,x_0,x_0,0)$ 
        & \texttimes \\
		
		\hline

		   &  & $\lambda=-1$: 
        $\left((x_2,x_4,x_1,x_3),\begin{pmatrix}
        -x_0 & 0 & 0 & 0 \\
        0 & x_0 & 0 & 0 \\
        0 & -x_2 & 0 & 0 \\
        x_1 & 0 & -x_2 & -x_0 \\
        \end{pmatrix}\right)$\\
           & &  \\
        $\mathfrak r_{4,\lambda}$  & $\text{diag}(0,\lambda x_0,\lambda x_0,x_0)$ 
        & $\lambda=\frac12$:
        $\left((x_2,x_1,x_3,x_4),\begin{pmatrix}
        \frac{x_0}{2} & 0 & 0 & 0 \\
        0 & 0 & 0 & 0 \\
        cx_1 & (c-2)x_2 & \frac{x_0}{2}  & 0 \\
        (c-1)x_3 & 0 & (c-2)x_2 & x_0 \\
        \end{pmatrix}\right)$\\
           & & \text{with} $c \in \RR$, $c^2-c-1=0$ \\
        
           &  & $\lambda=1$: $\left((x_1,x_4,x_2,x_3),\begin{pmatrix}
        0 & 0 & 0 & 0 \\
        0 &  x_0 & 0 & 0 \\
        x_4 & 0 &  x_0 & 0 \\
        0 & 0 & 0 & x_0 \\
        \end{pmatrix}\right)$ \\
		\hline

		$\mathfrak d_4$  & $\text{diag}(x_0,-x_0,0,0)$ 
        & \texttimes \\
		
		\hline
		
		 &  
        & $\lambda=\frac12$:
        $\left( (x_1,x_4,x_2,x_3), 
        \begin{pmatrix}
        \frac{x_0}{2} & 0 & 0 & 0 \\
        0 & 0 & 0 & 0 \\
        0 & -x_1 & \frac{x_0}{2} & 0 \\
        - x_2 & 0 & -x_1 & x_0 \\
        \end{pmatrix}\right)$\\
        
        $\mathfrak d_{4,\lambda}$  & $\text{diag}(\lambda x_0,(1-\lambda)x_0,x_0,0)$ &  \\
          
		 &   & $\lambda=2$:
        $\left( (x_2,x_1,x_3,x_4), 
        \begin{pmatrix}
        -x_0 & 0 & 0 & 0 \\
        0 & 2x_0 & 0 & 0 \\
        c x_1 & (c-2) x_2 & x_0 & 0 \\
        (c-1) x_3 & 0 & (c-2) x_2 & 0 \\
        \end{pmatrix}\right)$\\
		& & \text{with} $c \in \RR$, $c^2 - c - 1 = 0$\\
		\hline
		
		$\mathfrak d'_{4,\lambda}$  &  
		$\begin{pmatrix}
        \lambda x_0 & x_0 & 0 & 0 \\
        -x_0 & \lambda x_0 & 0 & 0 \\
        0 & 0 & 2\lambda x_0 & 0 \\
        0 & 0 & 0 & 0 \\
        \end{pmatrix}$ 
        & \texttimes\\
		
		\hline 
		
		\bottomrule
	\end{tabular} 
	\caption{Simply transitive actions for $\ggo$ with nilshadow $\mathfrak n(\ggo')=\hgo_3\oplus \RR$}
	\label{Nilshadow-h3R}
    \end{table}

\

	\begin{table}[h!]
	\def\arraystretch{1.7}
	\centering
	\begin{tabular}{l|c|c}
    $\ggo$ & $\hgo_3\oplus\RR\rtimes\Der(\hgo_3\oplus\RR)$ & $\Der(\ngo_4)$ \\
		\hline
       $\mathfrak{r}_{4}$  &  $\left( (x_1,x_2,x_4,x_3), 
        \begin{pmatrix}
        0 & 0 & 0 & 0 \\
        0 & x_0 & 0 & 0 \\
        0 & -x_1 & x_0 & x_1 \\
        0 & x_1 & 0 & x_0 \\
        \end{pmatrix}\right)$ & 
        $\begin{pmatrix}
        0 & 0 & 0 & 0 \\
        0 & x_0 & 0 & 0 \\
        0 & 0 & x_0 & 0 \\
        0 & 0 & 0 & x_0 \\
        \end{pmatrix}$ \\
		
		\hline
		
		$\mathfrak{h}_{4}$  & $\left( (x_1,x_3,x_4,x_2), 
        \begin{pmatrix}
        x_0 & 0 & 0 & 0 \\
        -x_2 & x_0 & 0 & 0 \\
        0 & 0 & 2x_0 & 0 \\
        0 & 0 & 0 & 0 \\
        \end{pmatrix}\right)$ & $\begin{pmatrix}
        x_0 & 0 & 0 & 0 \\
        0 & 0 & 0 & 0 \\
        0 & 0 & x_0 & 0 \\
        0 & 0 & 0 & 2x_0 \\
        \end{pmatrix}$ \\

		\hline
		
		\bottomrule
	\end{tabular} 
	\caption{Simply transitive actions for $\ggo$ with nilshadow $\mathfrak n(\ggo')=\mathfrak{n}_4$ }
	\label{Nilshadow-n4}
\end{table}

\begin{table}[h!]
	\def\arraystretch{1.5}
	\centering
	\begin{tabular}{l|c|c|c}
		$\ggo$ & $\mathfrak n(\ggo')$ & $\hgo_3\oplus \RR$ & $\ngo_4$ \\
		\hline
		$\mathfrak r\mathfrak h_3$ & $\mathfrak r\mathfrak h_3$ & \checkmark & \checkmark \\
		
		\hline
		
		$\mathfrak n_4$ & $\mathfrak n_4$ & \checkmark & \checkmark \\
		
		\hline
		
		$\mathfrak r\mathfrak r_3$  & $\mathfrak r\mathfrak h_3$ & \checkmark & \texttimes \\
		
		\hline
		
		$\mathfrak r\mathfrak r_{3,\lambda}$  & $\RR^4$& $\lambda \in \left\{  \pm1, 0 \right\}$ & \texttimes \\
		
		\hline

		$\mathfrak r\mathfrak r'_{3,\lambda}$ & $\RR^4$& $\lambda=0$ & \texttimes\\
		
		\hline
		
		$\mathfrak r_2\mathfrak r_2$  &  $\RR^4$& \texttimes & \texttimes\\
		
		\hline
		
		$\mathfrak r'_2$  &  $\RR^4$& \texttimes & \texttimes\\
		
		\hline
		
		$\mathfrak r_4$  & $\ngo_4$ & \checkmark & \checkmark \\
		
		\hline
		
		$\mathfrak r_{4,\lambda}$  & $\mathfrak r\mathfrak h_3$& \checkmark & $\lambda \in \left\{\pm1,\frac12\right\}$ \\
		
		\hline
		
		$\mathfrak r_{4,\mu,\lambda}$  & $\RR^4$ & for specific values of $(\mu,\lambda)$ & for specific values of $(\mu,\lambda)$\\
		
		\hline
		
		$\mathfrak r'_{4,\gamma,\delta}$ & $\RR^4$& $\delta=0$ or $\gamma=2\delta$ & \texttimes \\

		\hline
		
		$\mathfrak d_4$  & $\mathfrak r\mathfrak h_3$& \checkmark & \texttimes \\
		
		\hline
		
		$\mathfrak d_{4,\lambda}$ & $\mathfrak r\mathfrak h_3$&  \checkmark & $\lambda \in \left\{ \frac12, 2\right\}$\\
		
		\hline
		
		$\mathfrak d'_{4,\lambda}$ & $\mathfrak r\mathfrak h_3$ &  \checkmark & \texttimes\\
		
		\hline 
		
		$\mathfrak h_{4}$ & $\mathfrak n_4$& \checkmark & \checkmark \\
		
		\hline
		\bottomrule
	\end{tabular} 
	\caption{Overview of existence of simply transitive actions from $\ggo$ to $\hgo_3\oplus \RR$ and $\ngo_4$}
		\label{table:overview}
\end{table}

 \clearpage
	\bibliography{ref}
	\bibliographystyle{plain}

\end{document}